\documentclass[a4paper,11pt,reqno,english]{amsart}
\usepackage{anysize,color} \usepackage[utf8]{inputenc}
\usepackage[dvipsnames]{xcolor}
\usepackage[colorlinks=true,urlcolor=blue,linkcolor={magenta},citecolor={orange}]{hyperref}  
\usepackage{float}
\usepackage{amsfonts,amssymb,amsmath}
\usepackage{amsthm}
\usepackage{url}
\usepackage{paralist}
\usepackage[mathscr]{euscript}
\usepackage{charter, babel}
\usepackage{bbold}
\usepackage{float,graphicx}
\usepackage[arrow,curve,matrix,tips,2cell, all]{xy}

\theoremstyle{plain}
\newtheorem{theorem}{Theorem}[section]
\newtheorem{lemma}[theorem]{Lemma}

\newtheorem{corollary}[theorem]{Corollary}

\newtheorem*{theorem*}{Theorem}
\newtheorem*{corollary*}{Corollary}

\newtheorem*{claim*}{Claim}
\newtheorem*{lemma*}{Lemma}

\theoremstyle{definition}
\newtheorem{definition}[theorem]{Definition}
\newtheorem*{definition*}{Definition}

\newtheorem{example}[theorem]{Example}

\newtheorem{question}[theorem]{Question}

\newcommand\R{\mathbb{R}}
\newcommand{\Ecal}{\mathcal{E}}
\DeclareMathOperator{\NashEquilibria}{Neq}

\allowdisplaybreaks

\author[Blagojevi\'c]{Matija Blagojevi\'{c}}
\thanks{ 
	The research by Matija Blagojevi\'c leading to these results has received funding from Studienstiftung des deutschen Volkes through their Promotionsstipendium program.
}
\email{matib03@zedat.fu-berlin.de}
\address{Freie Universit\"{a}t, 14195 Berlin, Germany}
\author[Sch\"{u}tte]{Christof Sch\"{u}tte}
\email{schuette@zib.de}
\address{Zuse Institut, 14195 Berlin, Germany}
\title{Existence of Multilateral Nash equilibria for families of games}

\begin{document}
\begin{abstract}
	This paper introduces two fundamentally new concepts to game theory: multilateral Nash equilibria and families of games. Starting with non-cooperative games in normal form, we show how these notions together seamlessly integrate into and naturally extend the classical theory, and simultaneously enable us to prove a powerful (multilateral) Nash equilibrium existence result with minimal assumptions on the non-cooperative game.

	Classically, a Nash equilibrium of a non-cooperative game is a global strategy with the property that whichever player unilaterally deviates from the equilibrium, also reduces his own profit. For a $k$-lateral Nash equilibrium we now require that whichever group of $k$ players collectively changes their strategies, also reduces all of the deviating players' profits. In this way, we obtain a filtration of equilibria, where the higher-lateral equilibria are less frequent. Intuitively, higher-lateral equilibria are increasingly stable against formations of coalitions and even more strongly discourage deviations. Furthermore, we derive an existence criterion for multilateral Nash equilibria and demonstrate how it reflects the increasing rarity of higher-lateral equilibria.

	Additionally, we show that some classical games have higher-lateral Nash equilibria, which in every case reveal the structure of these games in a descriptive and intuitive way, from a new point of view.

	Many external conditions may influence a game. A family of games is a parameterized collection of non-cooperative games, where the parameter affects every aspect of the game. Typically, we assume that this dependence is continuous, thereby introducing a new structure. That way, we can avoid analyzing the games one at a time, and instead treat the family as a whole. This allows the parameter to take a central role in our theory, and shifts our attention from seeking a special strategy to searching for a special game with preferred strategies.

	Though multilateral equilibria are increasingly rare in one particular game, our main result proves the existence of a multilateral equilibrium in a family of games. By exploiting the availability of the parameter space, we are able to maintain minimalistic assumptions on the games individually and still find a multilateral equilibrium. Surprisingly, the clique covering number of the Kneser graph makes a central appearance.

	To illustrate the relevance of the newly introduced notions we conclude with two applications. The first shows that there exist games with multilateral equilibria in the family of finite games, thereby opening the door for new fundamental research possibilities in this particularly intensely researched subfield of game theory. The second application expands the classical work of Nobel laureates Kenneth Arrow \& G\'erard Debreu from 1954, which modelled the state economy consisting of producers, consumers and the state itself. We show that the classical microeconomic assumptions are sufficient for the state to be able to ensure the existence of arbitrarily high-lateral Nash equilibria in its economy.
\end{abstract}
\maketitle

\vspace{-2em}

\section*{Introduction}

Imagine that you and your friend are trying to plan a trip together with two activities in mind, camping or going to an amusement park.
Since you are together on this trip, each one of you would be unhappy if you attend an activity alone.
However, both of you find the amusement park more engaging than going camping.
So what should you decide?

\medskip
For a moment, assume you cannot influence the activity choice of your friend.
In that case, to maximize enjoyment, your decision should mirror what your friend does: If he goes camping, go with him, if he opts for the amusement park, follow along.
If you cannot persuade your friend, there is no way for you to guarantee the maximum enjoyment.
This situation of the two comrades can be modelled as a non-cooperative game in normal form, a variant of the so-called date dilemma.

\medskip
Let $N\in \mathbb{N}$ be a fixed integer representing the number of players or participants in our non-cooperative game, labelled with the integers $\{1,\dots,N\}$.
In our example $N=2$.
Each player $i$ has a strategy space $E_i$ with $E:= E_1\times\cdots\times E_N$ being the global strategy space.
Thus, in our example we have that $E_1=E_1=\{C, A\}$, where $C$ stands for camping and $A$ for amusement park, and so the global strategy space is $E=\{(C,C),(C,A),(A,C),(A,A)\}$.
Player $i$ ``plays'' non-cooperative game by selecting a strategy $x_i\in E_i$.
An $n$-tuple $(x_1,\dots,x_n)\in E$ is called a global strategy.
 
\medskip
Each player $i$ has a payoff function $\theta_i\colon E \longrightarrow \R$, classically referred to as ``utility function'', ``target function'', ``profit function'', etc., which records his preferences.
Of course each player wants to maximize their payoff function.
In our example we set
\begin{displaymath}
	\theta_1=
	{\small
	\left( \begin{array}{cccc}
			(C,C) & (C,A) & (A,C) & (A,A) \\
			3     & 0     & 0     & 5
		\end{array} \right)
	}
	\qquad\text{and}\qquad
	\theta_2=
	{\small
	\left( \begin{array}{cccc}
			(C,C) & (C,A) & (A,C) & (A,A) \\
			3     & 0     & 0     & 5
		\end{array} \right)
	}
\end{displaymath}

\medskip
Thinking back to the case when you cannot influence your friend,
one of the classical notions for preferable decisions is a Nash equilibrium.
A global strategy $x\in E$ is called a Nash equilibrium, if for every player $i$ and for every possible \textit{unilateral} change in strategy $y_i\in E_i$ of player $i$, the payoff $\theta_i(x_1,\dots,y_i,\dots,x_N)$ which player $i$ receives is worse than the payoff $\theta_i(x_1,\dots,x_i,\dots,x_N)$, which he would've received in the Nash equilibrium. So
\begin{displaymath}
	\theta_i(y_i,x_{-i}):= \theta_i(x_1,\dots,y_i,\dots,x_N) \leq \theta_i(x_1,\dots,x_i,\dots,x_N),
\end{displaymath}
for all players $1\leq i\leq N$ and for all strategies $y_i\in E_i$. In short, we could define a Nash equilibrium as a global strategy such that no player can \textit{unilaterally} increase his profit.

\medskip
The fact that in a Nash equilibrium we only consider unilateral changes motivates the original name of \textit{non-cooperative} game, since the possibility of two or more players cooperating in agreement is ruled out.

\medskip
Evidently, our example has two Nash equilibria, which are the global strategies where both friends agree, so $(C,C)$ and $(A,A)$. If in these global strategies either of the friends changes his strategy, he will reduce his own enjoyment.

\medskip
If the friends wish to get the maximum possible enjoyment, they both need to simultaneously agree on $(A,A)$. To achieve this, they may need to undertake a \textit{bilateral} change in strategies. The reality that sometimes the all-around best strategy requires the cooperation of multiple players motivates the definition and study of a \textit{higher-lateral} Nash equilibrium.
In other words, we are interested in global strategies $x\in E$ such that, for example, for any subset of two (distinct) players $I:=\{i_1, i_2\}\subseteq [N] :=\{1,\dots,N\}$ and any two strategies $(y_{i_1}, y_{i_2})\in E_{i_1}\times E_{i_2}$, the payoff of either player is worse, that is
\begin{displaymath}
	\theta_{i_1}(y_I,x_{-I}):=\theta_{i_1}(x_1,\dots,y_{i_1},\dots,y_{i_2},\dots,x_N) \leq \theta_{i_1}(x_1,\dots,x_i,\dots,x_N),
\end{displaymath}
or
\begin{displaymath}
	\theta_{i_2}(y_I,x_{-I}):=\theta_{i_2}(x_1,\dots,y_{i_1},\dots,y_{i_2},\dots,x_N) \leq \theta_{i_2}(x_1,\dots,x_i,\dots,x_N).
\end{displaymath}
We call such a global strategy $x$ a \textit{bilateral} or $2$-\textit{lateral} Nash equilibrium.
In an analogous way for any integer $k$, with $1\leq k\leq N$, we can define a $k$-lateral Nash equilibrium.
This way we obtain a descending filtration over the set of all Nash equilibria.
A natural question to ask is whether for a (class of) non-cooperative game there exist $k$-lateral Nash equilibria for some $1\leq k\leq N$?

\medskip
In the first part of this paper we formally introduce multilateral Nash equilibria and demonstrate how many classical notions of non-cooperative game theory seamlessly carry over. 
We discuss the relevant best-reply correspondences, the Nikaido--Isoda function, and we derive existence criteria which are generalizations of the classical results for non-cooperative games, so in particular they include the case of the usual ($1$-lateral) Nash equilibrium. 

\medskip
Next, a notion of families of games is introduced.
At first glance, the definition is exactly what one may expect: A family of $N\in\mathbb{N}$-player games is a collection of non-cooperative games parameterized by some parameter space $B$. So in particular the strategy spaces and profit functions depend on the parameter $b\in B$.
Even with such a simple definition, almost all classical notions extend from non-cooperative games, now to families of games.
The most relevant case considered is when the family of games is not merely a parameterized collection, but rather a fiber bundle over the parameter space.
At this point, we exhibit how the family obtains a unifying structure with the parameter space playing a central role in the new theory.
We conclude this part by interconnecting the two notions we've introduced, and we prove numerous multilateral Nash equilibrium existence criteria for families of games. One particular highlight is the following result, Corollary~\ref{cor:Main_result_01_02}:

\begin{corollary*}
	Let $p$ be a prime, $N\geq 1$ an integer, and let $\Ecal/B:=(\Ecal^{(b)}:b\in B)$ be a family of $N$-player non-cooperative games with the following properties
	\begin{compactitem}[\quad --]
		\item $B$ is a compact $\mathbb{F}_p$-orientable manifold,
		\item $E_i$ is a fiber bundle over $B$ with compact total space $E_i$, which is a subbundle of a finite dimensional (real) vector bundle $V_i$ over $B$ for all $1\leq i\leq N$ with each fiber being convex,
		\item $E:=E_1\times_B\cdots \times_BE_N$ is the fiber product bundle over $B$ and is also a subbundle of the fiber product of vector bundle $V:= V_1\times_B\cdots\times_B V_N\cong V_1\oplus\cdots\oplus V_N$,
		\item $\theta_i^{(b)}(\cdot, x_{-i})\colon E_i\longrightarrow\mathbb{R}$, is quasi-concave for every $b\in B$, every $i\in [N]$ and every $x_{-i}\in E_{-i}$.
	\end{compactitem}

	\smallskip\noindent
	If the monomial $e(V_1)^{\xi(N,k)}\cdots e(V_N)^{\xi(N,k)}$ in $\mathbb{F}_p$ Euler classes of vector bundles $V_1,\dots, V_N$ does not vanish in the cohomology $H^*(B;\mathbb{F}_p)$, then there exists a parameter $b\in B$ such that the non-cooperative game $\Ecal^{(b)}$ from the family $\Ecal/B$ has a $k$-lateral Nash equilibrium.
	Here $\xi(N,k)$ denotes the clique covering number on the Kneser Graph $K(N,k)$.
\end{corollary*}

\medskip
It turns out that, as illustrated in the previous corollary, for one of our central results, Theorem \ref{thm:Main_result_02_02}, the condition on the parameter space involves the clique covering number $\xi(N,k)$ on the Kneser Graph $K(N,k)$, where $N$ is the number of players and $k$ is the laterality of the Nash equilibrium. 
This number arises thanks to the structure of the best-reply correspondence in the multilateral case, and is actually a natural generalization of a nearly obvious phenomenon in the case of $k=1$: The Kneser Graph $K(N,1)$ is the complete graph, so its clique covering number is $1$, which turns out to be the reason why there is just one ``global'' best-reply correspondence of a non-cooperative game in the classical case.

\medskip
Finally, we show two applications of the introduced notions and established existence results.
The first one begins with the observation that all $N$-player finite games constitute a single family of games. 
We then demonstrate that multilateral equilibria indeed do appear, and not infrequently: The family of $3$-player games contains infinitely many games with $2$-lateral equilibria. 

\medskip
The second application pays homage to the historical origins of game theory in economy, and is closely related to the famous work of Nobel laureates Kenneth Arrow \& G\'erard Debreu. 
One of the highlights of their result was the unification of three distinct economic theories, which before them were typically~\cite{HendersonQuandt1958, Baumol1965} considered individually: the theory of consumers, the theory of producers as well as the policies of the state. 
In Arrow and Debreu's model, the state has the special task of determining the prices of commodities such that commerce in the economy is maximized. 
The state is modelled as yet another interest-driven player who fairly competes with all the producers and consumers. Arrow and Debreu named the model they created ``Abstract Economy'', which became the namesake of abstract economies, also known as generalized Nash equilibrium problems. 
Their result is the first Nash equilibrium existence result of a --- or rather \textit{the} --- Abstract Economy.

\medskip
Our example is heavily inspired by the original Abstract Economy and also connects the classical theories of consumers and producers. 
However, we are also interested in the existence of multilateral Nash equilibria, as they have a very natural interpretation in economy: In a $k$-lateral Nash equilibrium, no group of $k$ players can gain any profit. 
Thus, if the state can achieve a high-lateral equilibrium, it can be confident that no rational group of producers or consumers would ever deviate from the equilibrium global strategy. 
This is why a high-lateral Nash equilibrium is desirable for the state: it ensures stability, predictability and resilience against deviations of larger and larger groups of players.

\medskip
Unlike Arrow and Debreu we follow a more traditional but more realistic assumption, that the market prices are determined by the behavior of the consumers and producers, instead of assuming that they are fixed by the state. 
For the prices, we use the Cournot competition rule, which says that prices decrease as supply increases.
As far as the state is concerned, we instead give it slightly different task: To ensure the availability of resources and appropriately plan the production. 
We show that in our model, it is possible for the state to create an economy which has an arbitrarily high-lateral Nash equilibrium assuming there are enough production possibilities and enough available resources. 
So, in these conditions, not only does there exist an equilibrium, but it can also be arbitrarily stable and resilient.

\medskip
Finally, the newly introduced concepts in this paper bring to light essential and complex open problems highly relevant to questions across various fields which could not have been studied before.

\part{Multilateral Nash equilibria}

In this section we formally introduce a new concept of $k$-lateral Nash equilibrium of an $N$-player non-cooperative game, where $1\leq k\leq N$ are integers.
An analogous notion can even be introduced in the context of abstract economies and repeated games. Initially, we opt for the simplest setup, so we can focus on the importance of the new concept while avoiding the complexities of the generalizations of non-cooperative games.
We will comment on the generalizations afterwards.

\section{Definition and basic properties}\label{sec:multilateral_nash_equilibria}

We start by introducing the notion of a non-cooperative game and recalling the definition of a Nash equilibrium.
For more details consult for example \cite[Sec.\,4.1]{Ichiishi_1983}.

\begin{definition}\label{def:game_in_normal_form}
	Let $N\geq 1$ be an integer.
	An $N$-\textit{player non-cooperative game}, $\Ecal=(E_i,\,\theta_i : 1\leq i\leq N)$ is given by the following data
	\begin{compactitem}[\quad --]
		\item a collection of $N$ (metric) spaces $(E_1,\dots,E_N)$, with $E:=E_1\times \cdots\times E_N$, and
		\item a collection of $N$ (continuous) functions $(\theta_1,\dots,\theta_N)$, where $\theta_i\colon E\longrightarrow \R$ for all $1\leq i\leq N$.
	\end{compactitem}
	In the case when at least one of the sets $E_i$ is empty, the associated game is called the {\em empty game}.
\end{definition}

Typically, $N$ stands for the number of players, $E_i$ is called the {\em set of strategies}, of the $i$-th player, and $\theta_i$ is the so-called {\em profit or utility function} of the player $i$.
The product $E$ is called the \textit{set of global strategies}.

\medskip
One of the central problems in the study of non-cooperative games is understanding the so-called Nash equilibria of the game. We begin with a complete definition.

\begin{definition}
	\label{def:NashEquilibrium}
	Let $N\geq 1$ be an integer and let $\Ecal=(E_i,\,\theta_i : 1\leq i\leq N)$ be an $N$-player non-cooperative game.
	A global strategy $x=(x_1,\dots,x_N)\in E$ is a {\em Nash equilibrium} of the game $\Ecal$ if for every $1\leq i\leq N$ it holds that
	\begin{displaymath}
		\theta_i(x) \,= \, \sup\big\{ \theta_i(x') : x'\in \{x_1\}\times \dots\times \{x_{i-1}\}\times E_i\times \{x_{i+1}\}\times \dots\times \{x_N\}\big\},
	\end{displaymath}
	or equivalently, for all $y_i\in E_i$
	\begin{displaymath}
		\theta_i(x) \, \geq \, \theta_i(x_1,\dots,x_{i-1},y_i,x_{i+1},\dots,x_N) =: \theta_i(y_i, x_{-i}).
	\end{displaymath}
	The set of all Nash equilibria of the game $\Ecal$ will be denoted by $\NashEquilibria(\Ecal)$.
\end{definition}

\medskip
Here we introduced a common notation in the context of Nash equilibria: If $b\in A_1\times\dots\times A_n$, $1\leq i\leq n$ and $a_i\in A_i$, then $(a_i, b_{-i}) \in A_1\times\dots\times A_n$ is the $n$-tuple obtained from $b$ by replacing the $i$-th coordinate with $a_i$. So $(a_i, b_{-i}) := (b_1,\dots,b_{i-1},a_i,b_{i+1}\dots,b_N)$.

\medskip
Restating the previous definition in words, a Nash equilibrium is a state of the game in which, whoever \textit{unilaterally} changes their strategy suffers a loss.
This definition can be seen as a strong argument for players to select their strategies to agree with the Nash equilibrium.
One knows that changing his own strategy away from the Nash equilibrium will result in a loss.

\medskip
To conclude, classically, the central problem in the study of $N$-player non-cooperative games has been explicitly phrased as follows:

\begin{question}
	Let $N\geq 1$ be an integer and let $\Ecal=(E_i,\theta_i : 1\leq i\leq N)$ be an $N$-player non-cooperative game.
	\begin{compactenum}[\rm \quad (A)]
		\item Is there a Nash equilibrium of the game $\Ecal$, or in other words, is $\NashEquilibria (\Ecal)\neq\emptyset$?
		\item If an equilibrium of the game $\Ecal$ exists, what can we say about the space $\NashEquilibria (\Ecal)\subseteq E$? 
			What is its topology, geometry, possible algebraic structure, and, in the case it is finite, what is its cardinality?
	\end{compactenum}
\end{question}

\medskip
Most of the classical work, as well as the state of the art, concentrates on studying unilateral changes and equilibria which are resistant to these changes.
A natural question to ask is:  Are there equilibria which are resistant to multilateral changes of players strategies?
To formalize our question we introduce the notion of $k$-lateral Nash equilibrium.

\begin{definition}\label{def:multi_lateral_nash_equilibrium}
	Let $N\geq 1$ and $1\leq k\leq N$ be integers and let $\Ecal=(E_i,\,\theta_i : 1\leq i\leq N)$ be an $N$-player non-cooperative game.
	A global strategy $x=(x_1,\dots,x_N)\in E$ is a {\em $k$-lateral Nash equilibrium} of the game $\Ecal$ if for every $k$ element subset $I=\{i_1,\dots,i_k\}\in \binom{[N]}{k}$ of $[N]$ and every index $i\in I$ it holds that
	\begin{multline*}
		\theta_i(x) \, = \, \sup\big\{ \theta_i(x') : \\
		x'\in \{x_1\}\times \dots\times \{x_{i_1-1}\}\times E_{i_1}\times \{x_{i_1+1}\}\times\dots\times  \{x_{i_k-1}\}\times E_{i_k}\times \{x_{i_k+1}\}\times \dots\times \{x_N\}\big\}.
	\end{multline*}
	Without loss of generality we assumed that $1\leq i_1<\dots<i_k\leq N$.
	Equivalently stated, $x$ is a $k$-lateral Nash equilibrium if for all $(y_{i_1},\dots,y_{i_k})\in E_{i_1}\times\dots\times E_{i_k}$ and all $i\in I$
	\begin{displaymath}
		\theta_i(x) \geq \theta_i(x_1,\dots,x_{i_1-1},y_{i_1},x_{i_1+1},\dots,x_{k_1-1},y_{k_1},x_{k_1+1},\dots,x_N) =: \theta_i(y_I, x_{-I}).
	\end{displaymath}
	The set of all $k$-lateral Nash equilibria of the game $\Ecal$ is denoted by $\NashEquilibria_k(\Ecal)$.
\end{definition}

\medskip
Here we expand on the notation introduced previously.
For $b\in A_1\times\dots\times A_n$ and $I:=\{i_1,\dots,i_k\}\subseteq [n]:=\{1,\dots,n\}$ we write $b_I := (b_{i_1},\dots,b_{i_k}) \in A_{i_1}\times\dots\times A_{i_k} =: A_I$.
This cleanly extends the common notation, since $b_{\{i\}} = (b_i) \in A_i$. Then for $b_I\in A_I$ we write
\begin{displaymath}
	(b_I,a_{-I}):= (a_1,\dots,a_{i_1-1},b_{i_1},a_{i_1+1},\dots,a_{i_k-1},b_{i_k},a_{i_k+1},\dots,a_N).
\end{displaymath}
Note that we implicitly assume that the order of entries in the $N$-tuple is maintained as given. This notation will come in handy throughout this manuscript.

\medskip
It is evident that every $k$-lateral Nash equilibrium is in particular a Nash equilibrium.
More generally, there is a descending filtration of the set of all Nash equilibria:
\begin{equation}\label{eq:filtration}
	\NashEquilibria(\Ecal)=\NashEquilibria_1(\Ecal) \ \supseteq \ \NashEquilibria_2(\Ecal) \ \supseteq \  \cdots \ \supseteq \ \NashEquilibria_N(\Ecal).
\end{equation}
It is clear that the existence of higher-lateral Nash equilibria is increasingly rare, so in particular the filtration above may be empty from some point on. On the other hand if the filtration is not trivial, the game has higher stability with respect to multiple players potentially cooperating towards alternate, preferable outcomes.

\section{Multilateral Equilibria Examples}
In order to illustrate the nature of multilateral equilibria we present a few classically motivated non-cooperative games.
The non-cooperative game in the first example will have a $(N-1)$-lateral equilibrium, but not an $N$-lateral equilibrium.

\begin{example} 
	\label{ex:GameWithNoKFoldNashEquilibrium}
	Consider the following game, played by $N$ players named $A,B_1,\dots,B_{N-1}$, with strategy spaces $E_{A} = E_{B_1} = \dots = E_{B_{N-1}} = [0,1]$.
	As usual set $E :=E_{A}\times E_{B_1}\times\dots\times E_{B_{N-1}}$, so in this case $E=[0,1]^N$ is an $N$-dimensional cube.
	Denote player $A$'s choice with $a\in E_A$ and player $B_i$'s choice by $b_i\in E_{B_i}$.
	Now,
	\begin{itemize}[\quad --]
		\item player $B_i$'s profit function is simply the projection function:
		      \begin{displaymath}
			      \theta_{B_i}\colon E\longrightarrow \mathbb{R}, \quad (a,b_1,\dots,b_{N-1}) \longmapsto b_i,
		      \end{displaymath}
		\item player $A$'s profit function is a truncated projection function defined by:
		      \begin{displaymath}
			      \theta_{A}\colon E\longrightarrow \mathbb{R}, \quad (a,b_1,\dots,b_{N-1}) \longmapsto \begin{cases}
				      a, \quad & \text{if } a + \max\{b_1,\dots,b_{N-1}\} < 1, \\
				      0, \quad & \text{else.}
			      \end{cases}
		      \end{displaymath}
	\end{itemize}
	In the case when $N=2$ this game is essentially a one-sided reformulation of the prisoner's dilemma.
	Intuitively, we could imagine this game representing a situation where $A$ is a suspect and the $B_i$ are presenting evidence against $A$.

	\medskip
	This game has a $(N-1)$-fold equilibrium at $(a,1,\dots,1) \in E$.
	We can easily check the definition, if only $N-1$ players are allowed to change their strategies, then either $A$ cannot change his strategy, and $B_i$ are not motivated to change, so they will remain at their choice of $b_1=\cdots=b_{N-1}=1$, or if $A$ is part of the change in strategy, then at least one of the $B_i$ doesn't get the opportunity to change away from $1$, and so $A$ cannot increase his profits.
	However, if all $N$ players are allowed to change their strategy, then there exist strategies where $A$ can improve his situation.

	\medskip
	This, to some degree, reflects the rather obvious reality that influencing more and more witnesses at the same time is increasingly difficult.
\end{example}

\medskip
The second example is a simple model of majority voting motivated by an example from the classical 1991 book of Drew Fudenberg and Jean Tirole~\cite{FudenbergTirole1991}, and it also has multilateral equilibria.

\begin{example}\label{ex:MajorityVoting}
	Assume an odd number $N=2n+1$ of players labelled $\{1,\dots,N\}$ are voting to elect a leader amongst them, so $E_i = \{1,\dots,N\}$ for $1\leq i\leq N$.
	The leader is determined as the candidate with the most votes, or, in case of a tie, the one with the smallest label amongst those who have the most votes.
	The tie-breaking rule is quite arbitrary and does not influence the conclusion we are about to present.
	It is natural to assume that each candidate would prefer to be elected as leader, and additionally he is indifferent amongst the other candidates. More explicitly, the profit functions are defined as follows
	\begin{displaymath}
		\theta_i(\textit{$j$\text{ \rm  elected as leader}}) = \begin{cases}
			1, \quad & i=j,         \\
			0, \quad & \text{else}.
		\end{cases}
	\end{displaymath}

	\medskip
	Interpreting the elected leader as the outcome of the game, we see that every possible outcome is a Nash equilibrium outcome, that is every candidate could be elected in a Nash equilibrium (strategy).
	However not all global strategies are Nash equilibria.
	For example, the global strategy $x$ given by $x_1=\dots=x_{n+1} = 2$ and $x_{n+2}=\dots=x_{2n+1} = 1$ is not a Nash equilibrium, because player $1$ could increase his profit by voting for himself instead and thereby gaining the majority.

	\medskip
	As for an example of a Nash equilibrium, the global strategy $x = (1,\dots,1)\in E$ is a Nash equilibrium. In fact, it is also an $n$-lateral Nash equilibrium, as no matter which $n$ players change their vote, they cannot change the majority of $N = 2n+1$ votes, so they cannot change the outcome of the game and therefore in particular cannot increase their profit functions.
	Of course, an $n+1$ or higher-lateral Nash equilibrium does not exist, because $n+1$ players would have the majority, but they couldn't agree on who to elect as a leader due to different preferences.

	\medskip
	However, there are also less-lateral Nash equilibria.
	E.g. the global strategy $x$ given by $x_1=\dots=x_{n+2} = 1$ and $x_{n+3}=\dots=x_{2n+1} = n+2$ is a $1$-lateral Nash equilibrium but not a $2$-lateral Nash equilibrium.
	No single change in strategy would change the elected leader in this case, however, if player $n+2$ and one other player changed their votes to $n+2$, then he would gain the majority.
	In general, the greater the agreement between the players, the higher-lateral the equilibrium is.

	\medskip
	This reflects the obvious phenomenon that officials who are elected by a large fraction of the voters are likely to remain elected for a long time.
	For a popular incumbent to be replaced, a large number of voters need to jointly change their vote --- exactly reflecting the higher-laterality of the Nash equilibrium.

	\medskip
	Thus, here the filtration of multilateral Nash equilibria has a direct manifestation in our model.
\end{example}

\medskip
Our last example is the so-called \textit{Cournot competition}, named after the French economist Augustin Cournot, which highlights how antagonism between players prevents the existence of a bilateral Nash equilibrium. Our presentation is based on the classical 1958 book \textit{Microeconomic Theory} by James M. Henderson and Richard E. Quandt~\cite{HendersonQuandt1958}.

\begin{example}\label{ex:CournotCompetition}
	Consider an oligopoly with $N$ firms producing the same homogeneous resource. Assume that every firm $i\in [N]$ can choose a production level $x_i\in [0,\infty)$, so $E_i:=[0,\infty)$. The market price is determined as some monotonously decreasing function of the total supply, say $p\colon [0,\infty)\longrightarrow[0,\infty)$ is a strictly monotonously decreasing function. Accordingly, the profit of each of the firms is given as
	\begin{displaymath}
		\theta_i(x) = \underbrace{x_i\cdot p(x_1+\dots+x_N)}_{\text{revenue}} - \underbrace{c_i(x_i)}_{\text{costs}},
	\end{displaymath}
	where $c_i\colon [0,\infty)\longrightarrow [0,\infty)$ is a firm-specific cost function representing the production cost associated with producing the homogeneous resource of the respective firm.

	\medskip
	In words, the Cournot competition can be summarized as follows: Each firm would like to sell as much of the resource as possible, but increased supply drives the sale price down.

	\medskip
	To simplify our analysis, we will restrict ourselves to the case when $N=2$ and to very simple price and cost functions given by
	\begin{displaymath}
		p(x) = 100 - \frac{1}{2}x, \qquad c_1(x_1) = 5x_1, \qquad c_2(x_2) = \frac{1}{2}x_2^2.
	\end{displaymath}
	Then the profits of the two players can be computed explicitly as
	\begin{displaymath}
		\theta_1(x_1,x_2)  = 95x_1 - \frac{1}{2}x_1^2 - \frac{1}{2}x_1x_2, \qquad	\theta_2(x_1,x_2) = 100x_2  - \frac{1}{2}x_1x_2 - x_2^2.
	\end{displaymath}
	We compute the best response of each player to the strategy of their opponent by maximizing the profit functions, given the opponent's strategy remains fixed. For this, we use the first-order necessary condition.
	\begin{displaymath}
		\frac{\partial\theta_1}{\partial x_1}(x_1,x_2)=0,\qquad \frac{\partial\theta_2}{\partial x_2}(x_1,x_2)=0,
	\end{displaymath}
	which is
	\begin{displaymath}
		95 - x_1 - \frac{1}{2}x_2=0,\qquad 100 - \frac{1}{2}x_1 - 2x_2=0.
	\end{displaymath}
	Denoting the first players best response to $x_2$ by $\Phi_1(x_2)$ and, accordingly, the second player's best response by $\Phi_2(x_1)$, we obtain that
	\begin{displaymath}
		\Phi_1(x_2) = 95 - \frac{1}{2}x_2, \qquad \Phi_2(x_1) = 50 - \frac{1}{4}x_1.
	\end{displaymath}
	Solving the linear system of equations and checking the second derivative condition yields that there exists only one Nash equilibrium $x^*=(x_1^*,x_2^*)=(80,30)$ and that
	\begin{displaymath}
		p(x_1^*,x_2^*)  = 45,   \qquad \theta_1(x_1^*,x_2^*)  = 3200, \qquad \theta_2(x_1^*,x_2^*) = 900.
	\end{displaymath}

	\medskip
	Now let us consider whether this Nash equilibrium is also a bilateral equilibrium. We compute the remaining partial derivatives:
	\begin{displaymath}
		\frac{\partial\theta_1}{\partial x_2}  = -\frac{1}{2}x_1, \qquad \frac{\partial\theta_2}{\partial x_1}  = -\frac{1}{2}x_2.
	\end{displaymath}
	At the equilibrium point $x^*=(80,30)$ the partial derivatives are both negative, so both players would prefer the other to reduce their production.
	Thus, the equilibrium point is not a bilateral equilibrium, because both players could get a better payoff after a bilateral change in strategies.
	For example, for player $1$, the global strategy $(y_1, y_2) = (95, 0)$, would yield a profit $\theta_1(y_1,y_2) = 4512.5$, which would be larger than the equilibrium profit of $3200$.
	Consequently, the Cournot competition does not have a bilateral Nash equilibrium.

	\medskip
	This highlights a common occurrence of non-existence of a bilateral equilibrium whenever there are two players whose interests are opposed to each other, meaning they can never reach a bilateral agreement.
\end{example}

\section{From multilateral Nash equilibria to multiple coincidences}\label{sec:MultipleCoincidences}

We base our initial steps towards developing the theory of multilateral Nash equilibria on the classical framework for studying non-cooperative games.

\subsection{Best-reply correspondence and existence of equilibria}

An essential object in the study of non-cooperative games is the so-called best-reply correspondence, a set-valued map describing the optimal responses of every player to any possible strategy from his opponents.
For background notions relating to the analysis of set-valued maps, see the classical book by Aubin and Frankowska~\cite{Aubin_Frankowska_2009}.
Throughout this section $N\geq 1$ is an integer and $\Ecal:=(E_i,\,\theta_i : 1\leq i\leq N)$ is a non-cooperative game.
We begin by recalling several definitions in the classical context.

\begin{definition}\label{def:MarginalFunction}
	The player $i$ {\em marginal function}, where $1\leq i\leq N$, is defined by:
	\begin{align*}
		\phi_i\colon E = E_1\times\dots\times E_N & \longrightarrow \mathbb{R}\cup\{\infty\} ,             \\
		x = (x_1,\dots,x_N)                       & \longmapsto \sup\{\theta_i(y_i,x_{-i}) : y_i\in E_i\}.
	\end{align*}
\end{definition}

\medskip
Typically, it will be assumed that the profit functions $\theta_i$ are bounded real functions, in which case the marginal functions $\phi_i\colon E\longrightarrow \mathbb{R}$ are also real functions.

\begin{definition}\label{def:BestReplyCorrespondence}
	The {\em  player $i$ best-reply correspondence}, where $1\leq i\leq N$, is given as follows:
	\begin{align*}
		\Phi_i\colon E =  E_1\times\dots\times E_N & \longrightarrow 2^{E_i} ,                                                 & \\
		x = (x_1,\dots,x_N)                        & \longmapsto \{y_i\in E_i : \theta_i(y_i, x_{-i}) = \phi(x_1,\dots,x_N)\}. &
	\end{align*}
\end{definition}

\medskip
We collect the players' best-reply correspondences $\Phi_i$ into a (global) best-reply correspondence for the whole game:
\begin{equation}\label{eq:def:BestReplyCorrespondence}
	\Phi\colon E\longrightarrow 2^{E}, \qquad x \longmapsto \Phi_1(x)\times \Phi_2(x)\times\dots\times\Phi_N(x).
\end{equation}

\medskip
One of the most important classical results shows that the Nash equilibria of a non-cooperative game are exactly the fixed points of its best-reply correspondence.

\begin{theorem}\label{thm:NashBestReplyCondition}
	A global strategy $x\in E$ is a Nash equilibrium if and only if $x\in\Phi(x)$, or in other words it is a fixed point of the best-reply correspondence $\Phi$.
\end{theorem}

\medskip
We will prove a generalization of this result for multilateral equilibria which includes this theorem as a particular case.

\medskip
For the remainder of this section let $k$ be an integer with $1\leq k\leq N$.
When defining the $k$-lateral marginal functions  some additional care needs to be taken.
As before $\Ecal=(E_i, \theta_i : 1\leq i\leq N)$ is a fixed non-cooperative game.

\begin{definition}\label{def:MultilateralMarginal}
	Let $I\in\binom{[N]}{k}$ be a subset of $k$ players. The {\em $k$-lateral marginal function} for the players $I$ is defined by
	\begin{align*}
		\phi_I\colon E = E_1\times\dots\times E_N & \longrightarrow (\mathbb{R}\cup\{\infty\})^{I}  ,                 \\
		x = (x_1,\dots,x_N)                       & \longmapsto (\sup\{\theta_i(y_I,x_{-I}) : y_I\in E_I\})_{i\in I},
	\end{align*}
	where $E_I:=\prod_{i\in I}E_i$.
	For the $i$-th coordinate of $\phi_I(x_1,\dots,x_N)$ we write $\phi_I(x_1,\dots,x_N)_i$ where $i\in I$.
\end{definition}

\medskip
A critical difference to the classical marginal function is that $\phi_I$ now takes values in $(\mathbb{R}\cup\{\infty\})^{I}$, instead of being a single real number or infinity.
This is because there is no preference towards any particular player.
Nonetheless, this does not complicate the definition of the best-reply correspondence.

\begin{definition}\label{def:MultilateralBestReplyCorrespondence}
	Let $I\in\binom{[N]}{k}$ be a subset of $k$ players.
	The {\em $k$-lateral best-reply correspondence} for the players $I$ is defined as
	\begin{align*}
		\Phi_I\colon E =  E_1\times\dots\times E_N & \longrightarrow 2^{E_I}  ,                                                                               & \\
		x = (x_1,\dots,x_N)                        & \longmapsto \big\{y_I\in E_I : (\forall i\in I) \ \theta_i(y_I, x_{-I}) = \phi_I(x_1,\dots,x_N)_i\big\}. &
	\end{align*}
\end{definition}

\medskip
With this definition we are now ready to state the generalization of Theorem~\ref{thm:NashBestReplyCondition} for multilateral Nash equilibria.

\begin{theorem}\label{thm:BestReplyConditionMultilateralNash}
	Consider the following set-valued maps:
	\begin{displaymath}
		F \colon E\xrightarrow{\Delta} \prod_{I\in\binom{[N]}{k}}E \xrightarrow{\prod \Phi_I} \prod_{I\in\binom{[N]}{k}} 2^{E_I}, \quad x\longmapsto (x)_{I\in\binom{[N]}{k}} \longmapsto \big(\Phi_I(x)\big)_{I\in\binom{[N]}{k}},
	\end{displaymath}
	\begin{displaymath}
		G \colon E\xrightarrow{\Delta} \prod_{I\in\binom{[N]}{k}}E \xrightarrow{\prod \pi_I} \prod_{I\in\binom{[N]}{k}} 2^{E_I}, \quad x\longmapsto (x)_{I\in\binom{[N]}{k}} \longmapsto \big(\{\pi_I(x)\}\big)_{I\in\binom{[N]}{k}}.
	\end{displaymath}
	Here $\Delta\colon E\longrightarrow \prod_{I\in\binom{[N]}{k}} E$ is the diagonal embedding and the maps $\pi_I\colon E\longrightarrow E_I$ are the canonical projections on the corresponding factors determined by the set $I$.

	\smallskip\noindent
	The global strategy $x\in E$ is a $k$-lateral Nash equilibrium of $\Ecal$ if and only $F(x)\cap G(x)\neq \emptyset$.
\end{theorem}

\begin{proof}
	Assume that $x\in E$ is a $k$-lateral Nash equilibrium. 
	By definition, we have for all $I\in\binom{[N]}{k}$, all $i\in I$ and all $y_I\in E_I$
	\begin{displaymath}
		\theta_i(x) \, = \, \sup\Big\{ \theta_i(y_I, x_{-I}) : y_I\in E_I \Big\},
	\end{displaymath}
	or equivalently, by definition of the marginal function $\phi_I$
	\begin{displaymath}
		\theta_i(x) = \phi_I(x)_i.
	\end{displaymath}
	Since the player $i\in I$ was arbitrary, we have
	\begin{displaymath}
		(\theta_{i_j}(x))_{i_j\in I} = \phi_I(x) \in (\mathbb{R}\cup\{\infty\})^{I}.
	\end{displaymath}
	Now, using the definition of $\Phi_I$ we conclude that
	\begin{displaymath}
		x_I = \pi_I(x)\in \Phi_I(x),
	\end{displaymath}
	exactly as claimed. All the arguments above can be reversed, so we've proven the claim of the theorem.
\end{proof}

\medskip
Notice that in the case $k = 1$ the set-valued map $G$ from the previous theorem becomes $G(x)=\{x\}$, so it is induced by the identity map $\operatorname{Id}\colon E\longrightarrow E$, meaning the statement of the theorem reduces exactly to Theorem~\ref{thm:NashBestReplyCondition}.

\medskip
The previous theorem can be rephrased by saying that there is a one-to-one correspondence
\begin{displaymath}
	\NashEquilibria_k \ \longleftrightarrow  \ \operatorname{Graph}(F)\cap\operatorname{Graph}(G)
\end{displaymath}
between $k$-lateral Nash equilibria of $\Ecal$ and elements of the intersection of the graph of $F$ and the diagonal $G$.

There is yet another way to define the multilateral best-reply correspondence which this time yields a simultaneous fixed-point criterion for a strategy $x\in E$ to be a $k$-lateral Nash equilibrium. Consider the following modification to the multilateral best-reply correspondence
\begin{displaymath}
	\overline{\Phi}_I\colon E\longrightarrow 2^{E}, \quad x\longmapsto \{(y_I, z_{-I})\in E \colon z_{-I}\in E_{-I},\ \theta_I(y_I, x_{-I}) = \phi_I(x)\}.
\end{displaymath}
Now all modified best-reply correspondences $\overline{\Phi}_I$ domains as well as codomains coincide, so it is possible to derive a simultaneous fixed-point existence criterion.
\begin{theorem}\label{thm:SimultaneousFixedPointCriterion}
	A $k$-lateral Nash equilibrium of $\Ecal$ is a point $x\in E$ such that $x$ is simultaneously a fixed point of all $\overline{\Phi}_I$, for $I\in\binom{[N]}{k}$.
\end{theorem}
\begin{proof}
	Assume that $x\in E$ is a $k$-lateral Nash equilibrium. Equivalently, for all $I\in\binom{[N]}{k}$, all $i\in I$ and all $y_I\in E_I$ we have that $\theta_i(y_I,x_{-I}) \leq \theta_i(x_{I}, x_{-I})$. Once again, as in Definition~\ref{def:multi_lateral_nash_equilibrium}, this is equivalent to $\theta_i(x_{I},x_{-I}) = \sup\{\theta_i(y_I,x_{-I})\colon y_I \in E_I\}$. Since this holds for all $i\in I$, from the definition of $\phi_I$ we obtain $\theta_I(x_I,x_{-I}) = \phi_I(x)$. But then, by definition of $\overline{\Phi}_I$ we have
	\begin{displaymath}
		x = (x_I, x_{-I}) \in \{(y_I, z_{-I})\in E \colon z_{-I}\in E_{-I},\ \theta_I(y_I, x_{-I}) = \phi_I(x)\} = \overline{\Phi}_I(x),
	\end{displaymath}
	for all $I\in\binom{[N]}{k}$, completing the proof.
\end{proof}

\begin{corollary}\label{cor:SimultaneousCoincidenceWithDiagonal}
	Assume the notation introduced so far: A $k$-lateral Nash equilibrium of $\Ecal$ exists if and only if
	\begin{displaymath}
		\Delta(E\times E)\cap \bigcap_{I\in \binom{[N]}{k}} \operatorname{Graph}(\overline{\Phi}_I) \neq \emptyset.
	\end{displaymath}
	Here $\Delta(E\times E):=\{(x,x)\colon x\in E\}$ is the diagonal in $E\times E$.
\end{corollary}

\medskip
Note that for $k=1$ the expression $\bigcap_{I\in \binom{[N]}{k}} \operatorname{Graph}(\overline{\Phi}_I)$ recovers exactly the graph of the classical best-reply correspondence $\operatorname{Graph}(\Phi)$. Hence, in this case the previous corollary reads: A Nash equilibrium exists if and only if $\Delta(E\times E)\cap \operatorname{Graph}(\Phi) \neq \emptyset$.

\subsection{Convexity and Upper Semicontinuity}
We analyze the best-reply correspondences $\Phi_I$, introduced in Definition  \ref{def:MultilateralBestReplyCorrespondence}, in the presence of assumptions of convexity.
We begin by recalling a classical lemma whose proof can be found in classical sources.
For completeness' sake we provide a proof.

\begin{lemma}\label{lem:QuasiConcaveMaximumSet}
	Let $K\subseteq X$ be a compact convex subset of a normed (real vector) space $X$, and let $f\colon K\longrightarrow\R$ be continuous and quasi-concave function.
	Then the set
	\begin{displaymath}
		M_f := \operatorname{argmax}(f) := \big\{x\in K\ : f(x)=\max\{f(x') : x'\in K\} \big\}
	\end{displaymath}
	is convex.
\end{lemma}

\begin{proof}
	Choose any two points $y_1, y_2\in M_f$ and any $0\leq t\leq 1$.
	Set $y_t=(1-t)y_1+ty_2$ to be a convex combination of $y_1$ and $y_2$.
	The convexity of $K$ implies that $y_t\in K$.
	On the other hand, the quasi-concavity of $f$, by definition, yields that $f(y_t)\geq \min\{f(y_1),f(y_2)\}$.
	Since $M_f$ is defined as the set of maxima of $f$ and $y_1,y_2\in M_f$, we get that necessarily $y_t\in  M_f$.
	Hence, $M_f$ is convex.
\end{proof}

We can use this lemma to establish the following property of the best-reply correspondence.
This is a generalization of the same lemma for the classical best-reply correspondence which is frequently used as a key step in proofs of the existence of classical Nash equilibria in non-cooperative games.

\begin{lemma}\label{lem:MultilateralBestReplyConvex}
	Let $\Ecal=(E_i,\,\theta_i : 1\leq i\leq N)$  be a non-cooperative game such that all strategy spaces $E_i$ are compact subsets of normed spaces and additionally all profit functions $\theta_i$ are quasi-concave.
	Then for all $I\subseteq N$ the best-reply correspondence $\Phi_I$ is convex-valued.
\end{lemma}
\begin{proof}
	It follows from the definitions of $\Phi_I$ and $\phi_I$ that

	\begin{align*}
		\Phi_I(x_1,\dots,x_N) = & \big\{y_I\in E_I : (\forall i\in I) \ \theta_i(y_I, x_{-I}) = \phi_I(x_1,\dots,x_N)_i\big\}                \\
		=                       & \bigcap_{i\in I}\big\{y_I\in E_I : \theta_i(y_I, x_{-I}) = \phi_I(x_1,\dots,x_N)_i\big\}                   \\
		=                       & \bigcap_{i\in I}\big\{y_I\in E_I : \theta_i(y_I, x_{-I}) = \sup\{\theta_i(y_I,x_{-I}) : y_I\in E_I\}\big\} \\
		=                       & \bigcap_{i\in I}  \operatorname{argmax} \theta_i(\cdot,x_{-I}) .
	\end{align*}
	In the last step we used the notation $\theta_i(\cdot,x_{-I})$ for the function $E_I\longrightarrow \mathbb{R}$ given by $y_I\longmapsto \theta_i(y_I,x_{-I})$, which is quasi-concave due to $\theta_i$ being quasi-concave by assumption.
	Now according to Lemma~\ref{lem:QuasiConcaveMaximumSet} each of the sets $\operatorname{argmax}\theta_i(\cdot,x_{-I})$ is convex and so their intersection $\Phi_I(x_1,\dots,x_N)$ has to be convex.
\end{proof}

\medskip
Next we discuss semicontinuity of the best-reply correspondence, once again mimicking classical results on non-cooperative games.
In the following we will make use of Claude Berge's famous Maximum theorem from 1959, which we state here.
For a proof as well as background notions on set-valued maps, see for example Frankowska and Aubin's classical book on set-valued analysis~\cite[Thm.\,17.31]{Aubin_Frankowska_2009}.
Also compare Berge's 1997 paper~\cite[Ch.VI.3,\,Thm.\,1]{Berge1997}.

\begin{theorem}
	\label{thm:MaximumTheorem}
	Let $X$ and $Y$ be metric spaces, the real function  $f\colon X\times Y\longrightarrow \mathbb{R}$ continuous, and let $F\colon X\longrightarrow 2^Y$ be a set-valued map that is strict, compact-valued, and both lower and upper semicontinuous.
	Then
	\begin{compactenum}[\rm \quad (1)]
		\item the marginal function $\varphi\colon X\longrightarrow \mathbb{R}$ defined by $\varphi(x):=\max\{f(x,y)\colon y\in F(x)\}$ is continuous, and
		\item the marginal set-valued map $\Phi\colon X\longrightarrow 2^Y$ given by $\Phi(x):= \{y\in F(x)\colon f(x,y)=\varphi(x)\}$ is upper semicontinuous and strict.
	\end{compactenum}
\end{theorem}

Another classical result of set-valued analysis is the so-called Closed graph lemma.
For a proof, see the classical book of Aubin and Frankowska~\cite[Thm.\,17.11]{Aubin_Frankowska_2009}.

\begin{lemma}\label{lem:ClosedGraphLemma}
	Let $F\colon X\longrightarrow 2^Y$ be a set-valued map.
	If $F$ is closed-valued, upper semicontinuous and has a closed domain in $X$, then the graph $\operatorname{Graph}(F)$ is closed in $X\times Y$.
	Additionally, if $Y$ is compact, then this statement is an equivalence.
\end{lemma}

Using Berge's Maximum theorem (Theorem~\ref{thm:MaximumTheorem}) and the Closed graph lemma (Lemma \ref{lem:ClosedGraphLemma}) we can prove the following lemma relating to the multilateral best-reply correspondences.

\begin{lemma}\label{lem:MultilateralBestReplyUpperSemicontinuous}
	Let $\Ecal=(E_i,\,\theta_i : 1\leq i\leq N)$  be a non-cooperative game such that all strategy spaces $E_i$ are compact metric spaces.
	Then for all $k\in [N]$ and for all $I\in\binom{[N]}{k}$, the best-reply correspondence $\Phi_I$ is an upper-semicontinuous set-valued map.
\end{lemma}
\begin{proof}
	Recall the definition of $\Phi_I$:
	\begin{displaymath}
		\Phi_I(x) =\big\{y_I\in E_I : (\forall i\in I) \ \theta_i(y_I, x_{-I}) = \phi_I(x_1,\dots,x_N)_i\big\}.
	\end{displaymath}
	Like in the previous proof, we have equivalently
	\begin{displaymath}
		\Phi_I(x) = \bigcap_{i\in I}\big\{y_I\in E_I : \theta_i(y_I, x_{-I}) = \phi_I(x_1,\dots,x_N)_i\big\}.
	\end{displaymath}
	For ease of notation, for $i\in I$ abbreviate as follows
	\begin{displaymath}
		\Phi_{I,i}(x):= \big\{y_I\in E_I : \theta_i(y_I, x_{-I}) = \phi_I(x_1,\dots,x_N)_i\big\}.
	\end{displaymath}
	Hence, $\Phi_I(x) = \bigcap_{i\in I} \Phi_{I,i}(x)$ and, in particular, $\operatorname{Graph}(\Phi_I) = \bigcap_{i\in I}\operatorname{Graph}(\Phi_{I,i})$.
	Now consider the continuous function
	\begin{displaymath}
		\theta_i(\cdot, x_{-I}) \colon E_I \longrightarrow \mathbb{R}, \qquad y_I\longmapsto \theta_i(y_I, x_{-I}).
	\end{displaymath}
	We use Berge's Maximum theorem as follows.
	For that, in the notation of  Theorem~\ref{thm:MaximumTheorem}, take that $X = E$, $Y = E_I$ and $f = \theta_i(\cdot, x_{-I})$.
	The map $F \colon E\longrightarrow 2^{E_{I}}$ is the constant set-valued map $x \longmapsto E_{I}$ with graph $\operatorname{Graph}(F) = E \times E_I$ which is evidently continuous and compact-valued.
	Substituting these choices in the definition of the marginal function and marginal correspondence in Theorem~\ref{thm:MaximumTheorem}, we have that the marginal function of $\theta_i(\cdot, x_{-I})$ is exactly $\phi_I(x_1,\dots,x_N)_i$, and the marginal correspondence is indeed $\Phi_{I,i}$.
	Thus, $\Phi_{I,i}$ is upper-semicontinuous.

	\smallskip
	From Lemma~\ref{lem:ClosedGraphLemma} follows that $\operatorname{Graph}(\Phi_{I,i})\subseteq E \times E_I$ is closed.
	Since the intersection of closed sets is closed, $\operatorname{Graph}(\Phi_I) = \bigcap_{i\in I}\operatorname{Graph}(\Phi_{I,i})$ is also closed.
	Once again employing Lemma~\ref{lem:ClosedGraphLemma} we obtain that $\Phi_I$ is upper-semicontinuous, exactly what we have set out to prove.
\end{proof}

\subsection{Nikaido--Isoda function and existence of equilibria}

A useful tool for the study of non-cooperative games and in particular abstract economies is the so-called Nikaido--Isoda Function, first introduced by Hukukane, Nikaido and Kazuo Isoda in 1955 in their publication~\cite{NikaidoIsoda1955}.
Instead of considering a family of marginal functions and deriving a family of best-reply correspondences, the Nikaido--Isoda function lets us determine the equilibria of a non-cooperative game, or even abstract economy, by understanding the marginal function of the Nikaido--Isoda function.

\medskip
We begin by recalling the classical theory, and then present a modification which allows us to generalize it to multilateral Nash equilibria.
Throughout this section we also maintain the notation we've used so far of $\Ecal:=(E_i,\theta_i: 1\leq i\leq N)$ for a non-cooperative game.

\begin{definition}\label{def:NikaidoIsodaClassical}
	The {\em Nikaido--Isoda} function of a non-cooperative game $\Ecal$ is given by
	\begin{align*}
		\Psi_c\colon E\times E & \longrightarrow \mathbb{R} ,                                                          \\
		(x,y)                  & \longmapsto \sum_{i\in[N]} \left(\theta_i(y_i,x_{-i}) - \theta_i(x_i, x_{-i})\right).
	\end{align*}
\end{definition}

We introduced the subscript $\Psi_c$ to refer to the classical definition of the Nikaido--Isoda function, because we will soon need to modify it and we will denote that modification by $\Psi$.
Before that, let us mention the main result concerning the classical Nikaido--Isoda function.

\begin{theorem}\label{thm:NikaidoIsodaClassicalTheorem}
	Let $\Psi_c$ be the Nikaido--Isoda function of a non-cooperative game $\Ecal$ and define the marginal function $V_c$ associated to $\Psi_c$ as usual
	\begin{align*}
		V_c\colon E & \longrightarrow \mathbb{R} \cup \{\infty\},       \\
		x           & \longmapsto \sup\big\{\Psi_c(x,y) : y\in E\big\}.
	\end{align*}
	Then the global strategy $x\in E$ is a Nash equilibrium of $\Ecal$ if and only if $V_c(x) = 0$.
\end{theorem}

\medskip
For a proof of this classical result, see e.g. the original publication~\cite{NikaidoIsoda1955} and for the case of abstract economies, for example \cite[Thm.\,1]{FacchineiKanzow2007}.

\medskip
We aim to introduce an analogue of Nikaido--Isoda function for the context of multilateral Nash equilibria in such a way that an appropriate analogue of Theorem \ref{thm:NikaidoIsodaClassicalTheorem} holds.
For that we begin by introducing a modification of the classical Nikaido--Isoda function.

\begin{definition}
	The {\em  modified Nikaido--Isoda function} of a non-cooperative game $\Ecal$ is given by:
	\begin{align*}
		\Psi\colon E\times E & \longrightarrow \mathbb{R} ,                                                          \\
		(x,y)                & \longmapsto \max_{i\in[N]} \left(\theta_i(y_i,x_{-i}) - \theta_i(x_i, x_{-i})\right).
	\end{align*}
\end{definition}

\medskip
An analogous result to Theorem~\ref{thm:NikaidoIsodaClassicalTheorem} holds for the modified Nikaido--Isoda function.

\begin{theorem}\label{thm:ModifiedNikaidoIsoda}
	Let $\Psi$ be the modified Nikaido--Isoda function of the non-cooperative game $\Ecal$.
	Define the marginal function associated to $\Psi$ in the usual way:
	\begin{align*}
		V\colon E & \longrightarrow \mathbb{R} \cup \{\infty\}, \\
		x         & \longmapsto \sup\{\Psi(x,y) : y\in E\}.
	\end{align*}
	Then the global strategy $x\in E$ is a Nash equilibrium of $\Ecal$ if and only if $V(x) = 0$.
\end{theorem}

\medskip
We do not yet prove the previous claim because it turns out to be a special case of a result which follows --- our generalization to multilateral Nash equilibria.
For that let us introduce the multilateral Nikaido--Isoda function.

\begin{definition}\label{def:MultilateralNikaidoIsoda}
	Let $N\geq 1$ and $1\leq k\leq N$ be integers and let $\Ecal=(E_i,\,\theta_i : 1\leq i\leq N)$ be an $N$-player non-cooperative game.
	The {\em $k$-lateral Nikaido--Isoda function} of $\Ecal$ is defined by
	\begin{align*}
		\Psi_k\colon E\times E & \longrightarrow \mathbb{R},                                                                                   \\
		(x, y)                 & \longmapsto \max_{I\in\binom{[N]}{k}}\max_{i\in I}\left(\theta_i(y_I, x_{-I}) - \theta_i(x_I, x_{-I})\right).
	\end{align*}
\end{definition}

\medskip
With the notion of $k$-lateral Nikaido--Isoda function introduced, we can formulate our generalization of the previous theorem.

\begin{theorem}\label{thm:MultilateralNikaidoIsodaTheorem}
	Let $\Psi_k$ be the $k$-lateral Nikaido--Isoda function of the non-cooperative game $\Ecal$.
	Define the marginal function associated to $\Psi_k$ in the usual way:
	\begin{align*}
		V_k\colon E & \longrightarrow \mathbb{R} \cup \{\infty\}, \\
		x           & \longmapsto \sup\{\Psi_k(x,y) : y\in E\}.
	\end{align*}
	Then the global strategy $x\in E$ is a $k$-lateral Nash equilibrium if and only if $V_k(x)=0$.
\end{theorem}
\begin{proof}
	Notice that $V_k(x) \geq 0$ for all $x\in E$. Indeed, an element of the defining supremum $\sup\{\Psi_k(x,y) : y\in E\}$ of $V_k(x)$ is in particular $\Psi_k(x,x)$, which we can compute
	\begin{displaymath}
		\Psi_k(x,x) =\max_{I\in\binom{[N]}{k}}\max_{i\in I}\left(\theta_i(x_I, x_{-I}) - \theta_i(x_I, x_{-I})\right) = \max_{I\in\binom{[N]}{k}}\max_{i\in I} 0 = 0.
	\end{displaymath}
	Hence, the supremum $V_k(x) \geq 0$ must at least $\Psi_k(x,x)=0$.

	\medskip
	Assume that $x\in E$ is a $k$-lateral Nash equilibrium. Then by definition, for every subset $I\in\binom{[N]}{k}$, all $i\in I$ and all $y_I\in E_I$ we have that
	\begin{displaymath}
		\theta_i(y_I,x_{-I}) \leq \theta_i(x_I,x_{-I}) \quad \Longleftrightarrow \quad \theta_i(y_I,x_{-I}) - \theta_i(x_I,x_{-I}) \leq 0.
	\end{displaymath}
	Thus,
	\begin{displaymath}
		\Psi_k(x,y)=\max_{I\in\binom{[N]}{k}}\max_{i\in I}\left(\theta_i(y_I, x_{-I}) - \theta_i(x_I, x_{-I})\right) \leq 0,
	\end{displaymath}
	and so $V_k(x)\leq 0$.
	Together with the previous observation this implies that $V_k(x)= 0$, as claimed.

	\medskip
	For the converse, suppose that $V_k(x)=0$ for some $x\in E$.
	In other words,
	\begin{displaymath}
		\sup\{\Psi_k(x,y) : y\in E\}=0.
	\end{displaymath}
	Hence, for every $y\in E$ it holds that
	\begin{displaymath}
		\Psi_k(x,y) =\max_{I\in\binom{[N]}{k}}\max_{i\in I}\left(\theta_i(y_I, x_{-I}) - \theta_i(x_I, x_{-I})\right) \leq 0.
	\end{displaymath}
	So in particular, for all $I\in\binom{[N]}{k}$ and all $i\in I$, we have that $\theta_i(y_I, x_{-I}) - \theta_i(x_I, x_{-I})\leq 0$, or equivalently $\theta_i(y_I, x_{-I}) \leq\theta_i(x_I, x_{-I})$, which is exactly the definition of a $k$-lateral Nash equilibrium
\end{proof}

\medskip
Notice that in particular Theorem \ref{thm:MultilateralNikaidoIsodaTheorem} reduces to Theorem~\ref{thm:ModifiedNikaidoIsoda} for $k=1$.

We've already noted that there is a filtration of multilateral Nash equilibria, meaning that the set of $k+1$-lateral Nash equilibria is always a subset of the $k$-lateral Nash equilibria for $1\leq k < N$. We demonstrate a manifestation of this in the following lemma.

\begin{lemma}\label{lem:NikaidoIsodaIncreasing}
	Let $1 \leq k < l \leq N$ and $x\in E$. Then $V_k(x) \leq V_l(k)$.
\end{lemma}
\begin{proof}
	We claim that for every $y\in E$ there exists $y^*\in E$ such that $\Psi_l(x,y^*) \geq \Psi_k(x,y)$. This would mean that every element of $\{\Psi_k(x,y)\colon y\in E\}$ is smaller than an element of $\{\Psi_l(x,y)\colon y\in E\}$, so $V_l = \sup\{\Psi_l(x,y)\colon y\in E\}$ must be larger than $V_k(x) = \sup \{\Psi_k(x,y)\colon y\in E\}$.

	Indeed, let $y\in E$ be arbitrary and consider
	\begin{displaymath}
		\Psi_k(x,y) =\max_{I\in\binom{[N]}{k}}\max_{i\in I}\left(\theta_i(y_I, x_{-I}) - \theta_i(x_I, x_{-I})\right) \leq 0.
	\end{displaymath}
	Since the maximum is finite, it must be attained for some fixed $i^* \in I^*$:
	\begin{displaymath}
		\Psi_k(x,y) = \theta_{i^*}(y_{I^*},x_{-I^*}) - \theta_{i^*}(x_{I^*},x_{-I^*}).
	\end{displaymath}
	Define $y^* \in E$ as $ y^* := (y_{I^*}, x_{-I^*})$ and choose any $J\binom{[N]}{l}$ with $I^*\subseteq J$. Notice that by construction of $y^{*}$ we then have $(y^{*}_{I^*}, x_{-I^*}) = (y^{*}_{J}, x_{-J})$. With this, we compute
	\begin{align*}
		\Psi_l(x,y^*) & = \max_{I\in\binom{[N]}{k}}\max_{i\in I}\left(\theta_i(y_I, x_{-I}) - \theta_i(x_I, x_{-I})\right) \\
		              & \geq \theta_{i^*}(y^{*}_J, x_{-J}) - \theta_i(y^*_J, x_{-J})                                       \\
		              & =  \theta_{i*}(y^{*}_{I^*}, x_{-I^*}) - \theta_i(y^*_{I^*}, x_{-I^*})                              \\
		              & = \Psi_k(x,y),
	\end{align*}
	as claimed.
\end{proof}

Recall that $V_k(x)\geq 0$ and Theorem~\ref{thm:MultilateralNikaidoIsodaTheorem}, which says that $x\in E$ is a $k$-lateral Nash equilibrium if and only if $V_k(x)=0$.
If $V_l \geq V_k$ for $1\leq k < l \leq N$, we see that $\NashEquilibria_k(\Ecal) = \{x\in E\colon V_k(x)=0\}\supseteq\{x\in E\colon V_l(x)=0\} = \NashEquilibria_l(\Ecal)$, recovering the filtration of multilateral Nash equilibria once again.

\medskip
Now, mirroring the classical approach, we introduce the best-reply correspondence associated to the newly defined multilateral
Nikaido--Isoda function.
Once more, we start by recalling the classical approach.

\begin{definition}\label{def:ClassicalNIBestReply}
	Let $\Psi_c\colon E\times E\longrightarrow \mathbb{R}$ be the classical Nikaido--Isoda function of $\Ecal$ with associated marginal function $V_c\colon E\longrightarrow\mathbb{R}\cup\{\infty\}$.
	The {\em Nikaido--Isoda best-reply correspondence} is the set-valued map
	\begin{align*}
		R_c\colon E & \longrightarrow 2^{E},                       \\
		x           & \longmapsto \{y\in E : \Psi_c(x,y)=V_c(x)\}.
	\end{align*}
\end{definition}

Curiously, in this classical case, we've just recovered the best-reply correspondence $\Phi\colon E\longrightarrow 2^{E}$ associated to the non-cooperative game $\Ecal$.
\begin{theorem}\label{thm:NIBestReplyEqualBestReplyClassic}
	Keeping the previously made assumptions and notations, for all global strategies $x\in E$ it holds that $\Phi(x) = R_c(x)$.
\end{theorem}

\medskip
For the proof of this fact see \cite{Blagojevic2024} or \cite[Prop.\,3.4]{blagojevic2025topologygeneralizednashequilibrium}.
In plain words, in classical case, there is not so much benefit from Nikaido and Isoda approach.
We will soon see that this dramatically changes in the multilateral context.

\medskip
Thus, an immediate corollary of Theorem \ref{thm:NIBestReplyEqualBestReplyClassic}, using the global strategy Theorem~\ref{thm:NashBestReplyCondition} is that $x\in E$ is a Nash equilibrium if and only if $x$ is a fixed point of the correspondence $R_c$, that is $x\in R_c(x)$.

\medskip
As already mentioned, Theorem~\ref{thm:NIBestReplyEqualBestReplyClassic} does not extend to multilateral best-reply correspondences, however an analogue of its corollary does remain valid.
For that, we will introduce the multilateral Nikaido--Isoda best-reply correspondence, which is now actually distinct from the multilateral best-reply correspondence.

\begin{theorem}\label{thm:MultilateralNIBestReply}
	Let $N\geq 1$ and $1\leq k\leq N$ be integers and let $\Ecal=(E_i,\,\theta_i : 1\leq i\leq N)$ be an $N$-player non-cooperative game.
	Suppose that $\Psi_k$ and $V_k$ are defined as above.
	The $k$-lateral Nikaido--Isoda best-reply correspondence is defined to be
	\begin{align*}
		R_k\colon E & \longrightarrow 2^{E} ,                      \\
		x           & \longmapsto \{y\in E : \Psi_k(x,y)=V_k(x)\}.
	\end{align*}
	Then the global strategy $x\in E$ is a $k$-lateral Nash equilibrium if and only if $x\in R_k(x)$.
\end{theorem}

\begin{proof}
	Suppose that the global strategy  $x\in E$ is a $k$-lateral Nash equilibrium.
	According to Theorem~\ref{thm:MultilateralNikaidoIsodaTheorem}, this is equivalent to the equality $V_k(x)=0$.
	But then $V_k(x) = 0 = \Psi_k(x,x)$, and so $x\in R_k(x)$ by definition of $R_k$.

	\medskip
	On the other hand, assume that $x\in R_k(x)$.
	Then by definition $\Psi_k(x,x) = V_k(x)$, and so $V_k(x) = 0$.
	Again equivalence of Theorem~\ref{thm:MultilateralNikaidoIsodaTheorem} yields that now $x$ is a $k$-lateral Nash equilibrium.
\end{proof}

\medskip
Using this theorem, we can simplify our search for a $k$-lateral Nash equilibrium, because instead of looking for an intersection of the images of the set-valued maps
\begin{displaymath}
	F\colon E\longrightarrow \prod_{I\in\binom{[N]}{k}} 2^{E_I} \qquad \text{and} \qquad G\colon E\longrightarrow \prod_{I\in\binom{[N]}{k}}E_I,
\end{displaymath}
from Theorem~\ref{thm:BestReplyConditionMultilateralNash}, we can simply look for fixed points of the set-valued map $R_k\colon E\longrightarrow 2^{E}$.

\section{Remarks}

There are several observations regarding multilateral equilibria in non-cooperative games worth mentioning.

\subsection{Preference relations}
The preference relations introduced in this section are preorder relations. We discuss a few properties of preorder relations in general.

\medskip
A preorder $\rho$ on $A$ is a reflexive and transitive relation. It defines an equivalence relation $\sim_\rho$ on $A$ via $a\sim_\rho b$ if and only if $a\rho b$ and $b\rho a$. With this equivalence, the quotient set $A/\sim_{\rho}$ becomes a totally ordered set. Slightly abusing notation, we set $[a] \rho [a']$ if $a\rho a'$, where $[a]$ denotes the equivalence class of the element $a\in A$. Here $\rho$ denotes both the order relation on  $A/\sim_{\rho}$ and the preorder relation on $A$.

\medskip
Every profit function $\theta_i\colon E\longrightarrow\mathbb{R}$ induces a total preorder $\prec_i$ on $E$ via the pullback construction.
That is, for $x,y\in E$ we set
\begin{displaymath}
	x \prec_i y
	\qquad \Longleftrightarrow \qquad \theta_i(x) \leq \theta_i(y).
\end{displaymath}

\medskip
Note that it is possible that $x\prec_i y$ and $y\prec_i x$ while $x\neq y$.
Clearly, different profit functions may induce the same preorders.
Yet the entire game is completely characterized by the preorders $(\prec_1,\dots,\prec_N)$. In particular, whether $x\in E$ is a $k$-lateral Nash equilibrium depends only on the preorders.

\medskip
The collection of total preorders $(\prec_1,\dots,\prec_N)$ introduces an additional partial preorder on $E$ by setting
\begin{displaymath}
	x \prec y
	\qquad \Longleftrightarrow \qquad (\forall i\in[N])\ 	x \prec_i y.
\end{displaymath}
The partial preorder $\prec$ further induces an equivalence relation on $E$ via
\begin{displaymath}
	x \sim y \qquad \Longleftrightarrow \qquad   x \prec y \ \land \ y \prec x.
\end{displaymath}
In other words, if $x\sim y$ then all players are indifferent between strategies $x$ and $y$, and essentially the strategies are identical from the point of view of the profit functions.
Note that in the case this preorder is a total preorder, the profit functions $(\theta_1,\dots,\theta_N)$ are essentially the same, that is, they see any two strategies in the same way. So all the players have identical preferences.

\medskip
If we consider $E/\sim$ with the quotient order relation, then we can see that there can only ever be at most one $N$-lateral Nash equilibrium, and additionally, it is the unique maximum under the quotient order induced by $\prec$.
Indeed, assume that $x$ and $x'$ are two distinct $N$-lateral Nash equilibria.
Since $x$ is an $N$-lateral Nash equilibria then in particular $[x']\prec [x]$, but $x'$ is also an $N$-lateral Nash equilibrium and so $[x]\prec [x']$.
The relation $\prec$ is the quotient order induced by the preorder $\prec$ and so antisymmetry implies that $[x]= [x']$.

\medskip
If there are multiple maximal strategies with respect to $\prec$, then they are not $N$-lateral Nash equilibria, but they are Pareto-optimal. For more insight in the approach to non-cooperative games from the point of view of preference relations, refer to the literature on \textit{Pure Exchange Economies}, e.g. an introduction can be found in the classical book on microeconomy by Henderson and Quandt~\cite{HendersonQuandt1958} as well as the classical work by Tatsuro Ichiishi \cite{Ichiishi_1983}.

\subsection{The Nikaido--Isoda best-reply correspondence}

In the criterion of Theorem~\ref{thm:MultilateralNIBestReply}, we have shown that a $k$-lateral Nash equilibrium of a non-cooperative game is a fixed point of the Nikaido--Isoda best-reply correspondence $R_k$. We recall the definition here
\begin{align*}
	R_k\colon E & \longrightarrow 2^{E} ,                      \\
	x           & \longmapsto \{y\in E : \Psi_k(x,y)=V_k(x)\}.
\end{align*}
Even though this criterion seems rather simple, it is, in fact, not easy to verify using classical methods.
This mostly comes down to the multilateral Nikaido--Isoda function, $\Psi_k$, not having many convenient properties.
If the profit functions $\theta_i$ are assumed to be continuous and bounded, then so are the functions $\Psi_k$ continuous and bounded, as the maximum of continuous and bounded functions.
The Maximum Theorem~\ref{thm:MaximumTheorem} then yields that indeed $V_k$ is also continuous and $R_k$ is upper-semicontinuous.

\medskip
Unfortunately, this is where the ``nice properties'' end.
If we wished to use, for example, the classical Kakutani fixed point theorem, we are confronted with an issue.
First, recall the statement of Kakutani's result.

\begin{theorem}\label{thm:KakutaniFixedPointTheorem}
	Let $E\subseteq \mathbb{R}^n$ be a non-empty, compact, convex subset of Euclidean space. Assume that $F\colon E\longrightarrow 2^E$ is an upper-semicontinuous, convex-valued set-valued map. Then $F$ has a fixed point.
\end{theorem}

Now, to apply Kakutani's fixed point theorem, we would need for $R_k$ to be convex-valued, but in fact, it may not be convex-valued.
The usual argument for the classical Nikaido--Isoda function relies on the assumption that the profit functions $\theta_i$ are concave and that the Nikaido--Isoda function defined via the sum, as in Definition~\ref{def:NikaidoIsodaClassical}
\begin{align*}
	\Psi_c\colon E\times E & \longrightarrow \mathbb{R} ,                                                          \\
	(x,y)                  & \longmapsto \sum_{i\in[N]} \left(\theta_i(y_i,x_{-i}) - \theta_i(x_i, x_{-i})\right),
\end{align*}
is a sum of concave functions, and so also concave.
But the maximum of concave functions isn't necessarily concave.

\medskip
One could ask: Why not express the multilateral Nikaido--Isoda function via a sum as well?
Because then, the critical from Theorem~\ref{thm:MultilateralNikaidoIsodaTheorem} wouldn't hold: To prove that $V_k(x)=0$ implies that $x$ is a $k$-lateral equilibrium, we would have to show for any $I\in\binom{[N]}{k}$ that:
\[
	\max_{i\in I}\left(\theta_i(y_I, x_{-I}) - \theta_i(x_I, x_{-I})\right) \leq 0.
\]
But this term cannot be isolated from the sum, if $\Psi_k$ is expressed a sum of these terms. So we could not verify that $x$ really is a $k$-lateral equilibrium. The proof of the Nikaido--Isoda best-reply criterion in Theorem \ref{thm:MultilateralNIBestReply} relies on this theorem.

\part{Families of non-cooperative games}

\section{Definition and basic properties}\label{sec:familes_of_games}
We begin with a definition of a family of non-cooperative games in a very intuitive and naive way motivated by the presentation of families of vector spaces in the work of Atiyah \cite{Atiyah1967}.
The notions introduced and discussed in this section can also be extended to the context of abstract economies and repeated games.

\begin{definition}\label{def:FamilyOfNonCooperativeGamesNaive}
	Let $N \geq 1$ be an integer and let $B$ be some parameter (topological) space.
	Assume that for every point $b\in B$ we have an $N$-player non-cooperative game $\Ecal^{(b)}:=(E^{(b)}_i,\theta^{(b)}_i : 1\leq i\leq N)$.
	The collection $\Ecal/B:=(\Ecal^{(b)}:b\in B)$ of non-cooperative games is called a {\em family of non-cooperative games over} $B$.
\end{definition}

One possible interpretation of this definition, from the point of view of game theory, is that the parameter $b\in B$ is an extrinsic signal given to the game which influences both the payoffs and strategy spaces of all players.

\medskip
To extend the basic notions of non-cooperative games, we do not need any more structure than Definition~\ref{def:FamilyOfNonCooperativeGamesNaive} already provides.

\medskip
Let $E_i := \bigsqcup_{b\in B} E_i^{(b)}$ be the disjoint union of $E_i^{(b)}$ over $B$ for every $1\leq i\leq N$.
For a moment we do not assume any additional structure on the $E_i$'s.
The {\em global strategy space} $E$  of $\Ecal/B$ is the fiberwise product (over $B$) of the strategy spaces $E_i$, that is
\begin{displaymath}
	E := E_1 \times_B \dots \times_B E_N := \bigsqcup_{b\in B}  E_1^{(b)}\times\dots\times E_N^{(b)}=\bigsqcup_{b\in B}  E^{(b)} \subseteq E_1\times\dots\times E_N,
\end{displaymath}
where $ E^{(b)}:=E_1^{(b)}\times\dots\times E_N^{(b)}$.
Each of the objects we just introduced come with the natural projection to the parameter space $B$:
\begin{displaymath}
	\pi_i\colon E_i\longrightarrow B, \ e_i\longmapsto b,\qquad
	\pi\colon E\longrightarrow B, \ e\longmapsto b,
\end{displaymath}
for $e_i\in E_i^{(b)}\subseteq E_i$ and $e\in E^{(b)}\subseteq E$.

\medskip
Now, the families of the players' profit functions
\[
	(\theta_1^{(b)}\colon E^{(b)}\longrightarrow\R : b\in B),\dots, (\theta_N^{(b)}\colon E^{(b)}\longrightarrow\R : b\in B)
\]
define the {\em players profit functions} on $E$ by
\begin{displaymath}
	\theta_i \colon E \longrightarrow   \mathbb{R}, \qquad x \longmapsto  \theta_i^{(b)} (x),
\end{displaymath}
for $x\in E^{(b)}=E_1^{(b)}\times\dots\times E_N^{(b)}  = E$ and $1\leq i\leq N$.

\medskip
Accordingly, the marginal functions can be defined in a similar fashion.
For $b\in B$ let $\phi_i^{(b)} \colon E^{(b)}\longrightarrow\mathbb{R}$ be the marginal function of player $i$ in game $\Ecal^{(b)}$.
That is,
\begin{align*}
	\phi_i^{(b)}\colon E^{(b)} = E_1^{(b)}\times\dots\times E_N^{(b)} & \longrightarrow \mathbb{R}\cup\{\infty\}  ,                        \\
	x = (x_1,\dots,x_N)                                               & \longmapsto \sup\{\theta_i^{(b)}(y_i,x_{-i}) : y_i\in E_i^{(b)}\},
\end{align*}
see Definition~\ref{def:MarginalFunction}.
Then we define the {\em marginal functions} of the family of non-cooperative games $\Ecal/B$ to be
\begin{displaymath}
	\phi_i \colon E \longrightarrow   \mathbb{R}, \qquad x \longmapsto  \phi_i^{(b)} (x),
\end{displaymath}
for $x\in E^{(b)} \subseteq \bigsqcup_{b\in B}E^{(b)} = E$ and $1\leq i\leq N$.

\medskip
The notion of the best-reply correspondence can also be introduced in this setting.
Denoting the best-reply correspondence of the game $\Ecal^{(b)}$, as in Definition~\ref{def:BestReplyCorrespondence}, with $\Phi^{(b)} \colon E^{(b)}\longrightarrow 2^{E^{(b)}}$, we define the {\em best-reply correspondence} of the family of non-cooperative games $\Ecal/B$ by
\begin{equation}\label{eq:def_best_reply_familes}
	\Phi \colon E \longrightarrow 2^{E}, \quad x \longmapsto \Phi^{(b)} (x),
\end{equation}
for $x\in E^{(b)} \subseteq \bigsqcup_{b\in B}E^{(b)} = E$.

\medskip
The novel property of the best-reply correspondence of $\Ecal/B$ compared to the classical notion is that the following diagram commutes
\begin{equation}\label{eq:commutative_diagram}
	\xymatrix{
	E\ar[rr]^-{\Phi}\ar[dr]_{\widehat{\pi}} & & 2^E\ar[dl]^{\widetilde{\pi}}\\
	& 2^B &}
\end{equation}
where $\widehat{\pi}\colon E\longrightarrow 2^B$ is given by $e\longmapsto\{\pi(e)\}$ for $e\in E$, and $\widetilde{\pi}\colon 2^E\longrightarrow 2^B$ is given by $U\longmapsto \pi(U)$ for $U\subseteq E$.
In the terminology on fiber bundles we can say that the best-reply correspondence of $\Ecal/B$ is a fiberwise set-valued map.

\medskip
Following in the footsteps of the proof of Theorem~\ref{thm:NashBestReplyCondition} and having in mind definitions for the families on non-cooperative games one can show the following claim.

\begin{theorem}\label{thm:FiberwiseBestReplyCondition}
	Assume the notation introduced so far: There exists a parameter $b\in B$ such that the non-cooperative game $\Ecal^{(b)}$ from the family $\Ecal/B$ has a Nash equilibrium if and only if the best-reply correspondence $\Phi \colon E \longrightarrow 2^{E}$ of $\Ecal/B$  has a fixed point.
\end{theorem}
\begin{proof}
	If $e\in E$ is a fixed point of $\Phi$, then  $e\in\Phi^{(\pi(e))}(e)$, so by Theorem~\ref{thm:NashBestReplyCondition} we have that $e$ is a Nash equilibrium.
	On the other hand, if  $e\in E^{(b)}$ is a Nash equilibrium, then it is a fixed point of $\Phi^{(b)}$ and consequently, by construction, also of $\Phi$.
\end{proof}

\medskip
Like in the classical case of a single non-cooperative game the central questions for families of non-cooperative games are as follows.

\begin{question}
	Let $N\geq 1$ be an integer and let $\Ecal/B:=(\Ecal^{(b)}:b\in B)$ be a family of non-cooperative games over $B$.
	\begin{compactenum}[\rm \quad (A)]
		\item Is there a point $b\in B$ such that the non-cooperative game $\Ecal^{(b)}$ has a Nash equilibrium, or in other words, is there $b\in B$ with the property that $\NashEquilibria (\Ecal^{(b)})\neq\emptyset$?
		\item What can be said about the subspace $\{b\in B : \NashEquilibria (\Ecal^{(b)})\neq\emptyset\}$ of $B$?
		\item What can be said about the subspace  $\bigsqcup_{b\in B}\NashEquilibria (\Ecal^{(b)})\subseteq \bigsqcup_{b\in B}E^{(b)}$ of $E$?
	\end{compactenum}
\end{question}

\medskip
In order to study the questions we just proposed, some additional assumptions on the family of games $\Ecal/B:=(\Ecal^{(b)}:b\in B)$ need to be made.
First, we assume that the strategy spaces $E_i$ of all players are total spaces of fiber bundles over $B$.
Then, the global strategy space $E$ is indeed the fiber product $E_1\times_B\cdots\times_B E_N$ and so the best-reply correspondence $\Phi \colon E \longrightarrow 2^{E}$ maps each point $e\in E$ into a subspace of the fiber $\pi^{-1}(\{\pi(e)\})$ of $E$ over $\pi(e)$.
Second, let each $E_i$ be subbundle of the vector bundle $V_i$ over $B$.
So, in particular, the global strategy space $E$ is a subbundle of the (real) vector bundle $V_1\times_B\cdots\times_B V_N$.

\medskip
The assumptions we have made allow us to use the well established language and tools of bundle theory.
Many classical sources for this theory are available, like for example \cite{Atiyah1967, Husemoller1994}.
In this framework we derive further criterions for the existence of Nash equilibrium for families of games.

\medskip
From now on assume that $\Ecal/B:=(\Ecal^{(b)}:b\in B)$ is a family of games with the following properties
\begin{compactitem}[\quad --]
	\item $E_i$ is a fiber bundle over $B$, which is a subbundle of a finite dimensional (real) vector bundle $V_i$ over $B$ for all $1\leq i\leq N$,
	\item $E:=E_1\times_B\cdots \times_BE_N$ is the fiber product bundle over $B$ which is also a subbundle of the fiber product of vector bundle $V_1\times_B\cdots\times_B V_N\cong V_1\oplus\cdots\oplus V_N$,
	\item the marginal functions $\phi_i \colon E \longrightarrow \mathbb{R}$ are continuous with respect to the inherited topology on $E$ for all $1\leq i\leq N$,
	\item the best-reply correspondence $\Phi \colon E \longrightarrow 2^{E}$ is upper-semi continuous with respect to the topology on $E$.
\end{compactitem}
In this way the spaces are equipped with appropriate topologies and the maps are continuous with respect to such choices.
These assumptions are not restrictive since they cover all relevant situations. Note that in particular, if the rather mild assumptions of the Maximum Theorem~\ref{thm:MaximumTheorem} are satisfied, then that theorem implies that the marginal functions are continuous, so even this assumption is not particularly restrictive.

\medskip
To derive a new Nash equilibrium existence criterion we begin by considering the pullback of the bundle $E\times_B E$ along the projection map $\pi\colon E\longrightarrow B$, as in the following pullback square diagram
\begin{equation}\label{eq:pull-back}
	\xymatrix{
	\pi^*(E\times_B E)\ar[rr]^-{\Pi}\ar[d]^{\mu} & & E\times_B E\ar[d]^{\pi\times_B\pi}\\
	E\ar[rr]^-{\pi}& & B.
	}
\end{equation}
Recall that by the definition of the pullback
\begin{displaymath}
	\pi^*(E\times_B E) := \{(e, (e', e'')) \in E\times (E \times_B E): \pi(e) = (\pi \times_B \pi)(e',e'') \}.
\end{displaymath}
Since the best-reply correspondence $\Phi \colon E \longrightarrow 2^{E}$ is fiberwise, meaning it satisfies the property \eqref{eq:commutative_diagram}, it induces the set-valued section $\widetilde{\Phi}\colon E\longrightarrow 2^{\pi^*(E\times_B E)}$ of the bundle $\pi^*(E\times_B E)$ given by
\begin{displaymath}
	\widetilde{\Phi} (e):= \{e\}\times\{e\}\times  \Phi(e) \ \subseteq \ \mu^{-1}(\{e\})  \ \subseteq \ \pi^*(E\times_B E).
\end{displaymath}
The fact that $\widetilde{\Phi}$ is a set-valued section means that $\mu\big(\widetilde{\Phi} (e)\big)=\{e\}$.

\medskip
The (single-valued) section $\delta\colon E\longrightarrow \pi^*(E\times_B E)$ of the bundle $\pi^*(E\times_B E)$ given by $\delta(e)=(e,e,e)$ induces the set-valued section $\Delta\colon E\longrightarrow 2^{\pi^*(E\times_B E)}$ by $\Delta(e)=\{\delta(e)\}$.
This section can be called the diagonal section of the fiber product.

\medskip
Now the new criterion for the existence of Nash equilibria for families of games can be stated as follows.

\begin{theorem}\label{thm:criterion_01}
	Assume the notation introduced so far: There exists a parameter $b\in B$ such that the non-cooperative game $\Ecal^{(b)}$ from the family $\Ecal/B$ has a Nash equilibrium if and only if there exists a point $e\in \pi^{-1}(\{b\})$ such that $\Delta(e)\cap \widetilde{\Phi}(e)\neq\emptyset$, or in other words
	\begin{displaymath}
		\bigcup_{\epsilon\in E}\widetilde{\Phi} (\epsilon) \ \cap \ \bigcup_{\epsilon\in E} \Delta(\epsilon)\neq \emptyset.
	\end{displaymath}
\end{theorem}
\begin{proof}
	According to Theorem \ref{thm:FiberwiseBestReplyCondition}: there exists a parameter $b\in B$ such that the non-cooperative game $\Ecal^{(b)}$ from the family $\Ecal/B$ has a Nash equilibrium if and only if the best-reply correspondence $\Phi \colon E \longrightarrow 2^{E}$ of $\Ecal/B$  has a fixed point.
	In other words, there exists a point $e\in E$ such that $e\in \Phi(e)$ and so $(e,e,e)\in \Delta(e)\cap \widetilde{\Phi}(e)$.

	\smallskip\noindent
	Conversely, if $(w,u,v)\in \Delta(e)\cap \widetilde{\Phi}(e)$ then $(w,u,v)=(e,e,e)$ for some $e\in E$ and additionally  $(w,u,v)\in   \{e'\}\times\{e'\}\times  \Phi(e')$ for some $e'\in E$.
	Consequently, $(e,e,e)\in \{e'\}\times\{e'\}\times  \Phi(e')$, implying that $e=e'$ and $e\in \Phi(e)$, which means that the best-reply correspondence $\Phi$ has a fixed point.
	This completes the proof.
\end{proof}

\medskip
The final criterion is a direct consequence of the previous one and the use of the fiberwise map $\Pi\colon \pi^*(E\times_B E)\longrightarrow E\times_B E$ from the pull-back diagram \eqref{eq:pull-back}.

\begin{corollary}\label{cor:criterion_02}
	Assume the notation introduced so far: There exists a parameter $b\in B$ such that the non-cooperative game $\Ecal^{(b)}$ from the family $\Ecal/B$ has a Nash equilibrium if and only if
	\begin{displaymath}
		\Pi\big(\bigcup_{\epsilon\in E}\widetilde{\Phi} (\epsilon) \big) \ \cap \ \Pi\big(\bigcup_{\epsilon\in E} \Delta(\epsilon)\big)\neq \emptyset.
	\end{displaymath}
\end{corollary}
\begin{proof}
	The proof follows directly from the previous theorem and implications: If $(e',e'')\in \Pi\big(\bigcup_{\epsilon\in E} \Delta(\epsilon)\big)$ then $e'=e''=:e$, and if $(e,e)\in\{e'''\}\times \Phi(e''')$ then $e=e'''$ and $e\in \Phi(e)$.
\end{proof}

\medskip
Note that the set $\Pi\big(\bigcup_{\epsilon\in E}\widetilde{\Phi} (\epsilon) \big)$ consists of all graphs of best-reply correspondences of games $\Ecal^{(b)}$ gathered together into a subset of the fiber product $E\times_B E$.

\medskip
The criterion for the existence of Nash equilibrium we just obtained has a significant advantage for us, compared to the other ones.
It provides us the possibility to use the following intersection lemma which is assumed to be folklore.
For the precise statement and complete proof of the lemma see for example \cite[Lem.\,4.3]{Blagojevic2011}.
Unless otherwise stated, we assume an appropriate (co)homology theory, for example, the singular (co)homology theory with coefficients in $\mathbb{F}_p$, the finite field with $p$ elements.

\begin{lemma}\label{lem:intersection_lemma}
	Let $p$ be a prime, $k\geq 1$ an integer and let $B$ be a compact $\mathbb{F}_p$-orientable manifold.
	Suppose that $V\overset{p}{\longrightarrow} B$ is a real vector bundle over $B$ with the property that the $k$-th power of its  $\mathbb{F}_p$-Euler class does not vanish, that is $e(V)^k\neq 0$ in $H^*(B;\mathbb{F}_p)$.

	\smallskip\noindent
	Let $T_0,\dots, T_k$ be subsets of the total space of the vector bundle  $V$ with property that all homomorphisms
	\[
		\xymatrix{H^{\dim(B)}(B;\mathbb{F}_p)\ar[r]^-{(p|_{T_i})^*} & H^{\dim(B)}(T_i;\mathbb{F}_p)}
	\]
	induced by the restrictions of the projection map $p|_{T_i}$ are injective, where $0\leq i\leq k$.
	Then
	\[
		T_0\cap \dots \cap T_k\neq \emptyset.
	\]
\end{lemma}

\medskip
Combining the criterion from Theorem \ref{thm:criterion_01} and Lemma \ref{lem:intersection_lemma} we prove the first ever result on the existence of Nash equilibria for families of games.

\begin{theorem}\label{thm:Main_result_01}
	Let $p$ be a prime, $N\geq 1$ an integer, and let $\Ecal/B:=(\Ecal^{(b)}:b\in B)$ be a family of $N$-player non-cooperative games with the following properties
	\begin{compactitem}[\quad --]
		\item $B$ is a compact  $\mathbb{F}_p$-orientable manifold,
		\item $E_i$ is a fiber bundle over $B$ with compact total space $E_i$, which is a subbundle of a finite dimensional (real) vector bundle $V_i$ over $B$ for all $1\leq i\leq N$,
		\item $E:=E_1\times_B\cdots \times_BE_N$ is the fiber product bundle over $B$ which is also a subbundle of the fiber product of vector bundle $V:= V_1\times_B\cdots\times_B V_N\cong V_1\oplus\cdots\oplus V_N$,
		\item the best-reply correspondence $\Phi\colon E\longrightarrow 2^E$ of $\Ecal/B$ is compact-valued,
		\item the homomorphism in cohomology, induced by the projection map $\pi\times_B\pi\colon E\times_B E\longrightarrow B$,
		\[
			\xymatrix{
			H^{\dim(B)}(B;\mathbb{F}_p)\ar[rr]^-{(\pi\times_B\pi)^*} & & 	H^{\dim(B)}\big(\Pi\big(\bigcup_{e\in E}\widetilde{\Phi} (e) \big) ;\mathbb{F}_p\big)
			}
		\]
		is injective.
	\end{compactitem}

	\smallskip\noindent
	If the monomial $e(V)^2=e(V_1)^2\cdots e(V_N)^2$ in Euler classes of vector bundles $V_1,\dots, V_N$ does not vanish in the cohomology $H^*(B;\mathbb{F}_p)$, then there exists a parameter $b\in B$ such that the non-cooperative game $\Ecal^{(b)}$ from the family $\Ecal/B$ has a Nash equilibrium.
\end{theorem}
\begin{proof}
	The result follows from Corollary \ref{cor:criterion_02} and Lemma \ref{lem:intersection_lemma} by taking
	\begin{compactitem}[\quad --]
		\item  $T_0$ to be the subspace $\Pi\big(\bigcup_{\epsilon\in E}\widetilde{\Phi} (\epsilon) \big)$ of $V\times_N V$, and
		\item  $T_1$ to be the intersection of the disk bundle $D(V\times_N V)$ and the total bundle of the diagonal vector subbundle of $V\times_N V$.
	\end{compactitem}
	Here the disk bundle is chosen in such a way that it contains the set $\Pi\big(\bigcup_{\epsilon\in E}\widetilde{\Phi} (\epsilon) \big)$ in its interior.
	This is possible due to the assumed compactness of $E_1,\dots,E_N$ and $B$.
	In this way we have achieved that $T_1$ is compact and that
	\[
		\xymatrix{
			H^{*}(B;\mathbb{F}_p) \longrightarrow H^{*}(T_1;\mathbb{F}_p)
		}
	\]
	is an isomorphism (in all degrees).
\end{proof}

\medskip
Note that an assumption of $\mathbb{F}_p$-orientability of vector bundles $V_1,\dots,V_N$ is not necessary, because we are considering the square of the Euler class, or more precisely, the Euler class of the Whitney $V\oplus V$ which, in particular, is always orientable.

\medskip
A highlight of our theorem is that the assumptions on the individual games are rather permissive, especially compared to more classical assumptions in game theory.

\section{Multilateral Nash equilibria for families of non-cooperative games}\label{sec:multilateral_nash_equilibria_family}

The generalizations for a non-cooperative game which we have introduced so far can easily be extended to abstract economies and repeated games, and additionally, they can be combined seamlessly.
Therefore, in this section, we briefly demonstrate how the notions defined in Section~\ref{sec:multilateral_nash_equilibria} carry over to families of games and then we prove an existence result for multilateral Nash equilibria.

\medskip
Let $N \geq 1$ and $1\leq k\leq N$ be (fixed) integers, and let $\Ecal/B:=(\Ecal^{(b)}:b\in B)$ be a family of non-cooperative games over $B$.
For a subset $I := \{i_1,\dots,i_k\}\in\binom{[N]}{k}$ we still use the previously introduced notation of $E_I$, but this time to refer to the fiber product
\begin{displaymath}
	E_I := E_{i_1}\times_B \cdots \times_B E_{i_k} :=  \bigsqcup_{b\in B} E_{i_1}^{(b)} \times\cdots\times E_{i_k}^{(b)} = \bigsqcup_{b\in B} E_I^{(b)}.
\end{displaymath}

\medskip
The definitions of the multilateral marginal function and the multilateral best-reply correspondence, now for the family of games $\Ecal/B$, exactly mirror their respective definitions in Section~\ref{sec:MultipleCoincidences}, Definition~\ref{def:MultilateralMarginal} and Definition~\ref{def:MultilateralBestReplyCorrespondence}.
For completeness' sake, we state them now in the context of families of games.

\medskip
The marginal function of the family of games $\Ecal/B$, associated to the subset $I\in\binom{[N]}{k}$, is defined to be
\begin{align*}
	\phi_I^{(b)}\colon E^{(b)} = E_1^{(b)}\times\dots\times E_N^{(b)} & \longrightarrow (\mathbb{R}\cup\{\infty\})^I  ,                               \\
	x = (x_1,\dots,x_N)                                               & \longmapsto (\sup\{\theta_i^{(b)}(y_I,x_{-I}) : y_I\in E_I^{(b)}\})_{i\in I}.
\end{align*}
Once again we use the notation $E_I^{(b)}:=\prod_{i\in I}E_i^{(b)}$.
Accordingly, the $k$-lateral best-reply correspondences of $\Ecal/B$, associated to the subset $I\in\binom{[N]}{k}$, are given by
\begin{align*}
	\Phi_I^{(b)}\colon E^{(b)} =  E_1^{(b)}\times\dots\times E_N^{(b)} & \longrightarrow 2^{E_I^{(b)}} ,                                                                                            & \\
	x = (x_1,\dots,x_N)                                                & \longmapsto \big\{y_I\in E_I^{(b)} : (\forall i\in I) \ \theta_i^{(b)}(y_I, x_{-I}) = \phi_I^{(b)}(x_1,\dots,x_N)_i\big\}. &
\end{align*}
As before, we can combine the set-valued maps $\Phi_I$ into a fiberwise map:
\begin{align*}
	\Phi_I\colon E: =  E_1\times\dots\times E_N & \longrightarrow 2^{E_I} , & \\
	x = (x_1,\dots,x_N)                         & \longmapsto \Phi_I^b,
\end{align*}
for $x\in E^{(b)}\subseteq \bigsqcup_{b\in B} E^{(b)} = E$.

\medskip
The multilateral best-reply correspondences of the family $\Ecal/B$, by construction, are also fiberwise, just like the previously constructed best-reply correspondence $\Phi$ of the family $\Ecal/B$, see \eqref{eq:def_best_reply_familes}.
In other words, for every $I\in\binom{[N]}{k}$, the following diagram commutes.
\begin{displaymath}
	\xymatrix{
	E\ar[rr]^-{\Phi_I}\ar[dr]_{\widehat{\pi}} & & 2^{E_I}\ar[dl]^{\widetilde{\pi_I}}\\
	& 2^B &}
\end{displaymath}
where, as before, $\widehat{\pi}\colon e\longmapsto\{\pi(e)\}$ and we are slightly abusing notation to denote  $\widetilde{\pi}\colon U\longmapsto \pi(U)$ being the projection maps.

\medskip
The definition of modified best-reply correspondences from Theorem~\ref{thm:SimultaneousFixedPointCriterion} can be extended to families of games.
Indeed, for $I\in\binom{[N]}{k}$ we have
\begin{align*}
	\overline{\Phi}_I^{(b)}\colon E^{(b)} & \longrightarrow 2^{E^{(b)}} ,                                                                                                                        & \\
	x = (x_1,\dots,x_N)                   & \longmapsto \overline{\Phi}_I^{(b)}(x):=\{(y_I, z_{-I})\in E^{(b)} \colon z_{-I}\in E^{(b)}_{-I}, \ \theta_I^{(b)}(y_I, x_{-I}) = \phi_I^{(b)}(x)\}, &
\end{align*}
and so we can define
\begin{displaymath}
	\overline{\Phi}_I\colon E\longrightarrow 2^{E}, \qquad x \longmapsto \overline{\Phi}_I^{(b)}(x),
\end{displaymath}
for $x\in E^{(b)} \subseteq \bigsqcup_{b\in B}E^{(b)} = E$.
Now we get an analogue of Theorem~\ref{thm:SimultaneousFixedPointCriterion}, the following existence criterion.

\begin{theorem}\label{thm:FamilyOfGamesSimultaneousFixedPointCriterion}
	Let $\Ecal/B:=(\Ecal^{(b)}:b\in B)$ be a family of non-cooperative games over $B$.
	Then there exists a parameter $b\in B$ such that $\Ecal^{(b)}$ has a $k$-lateral Nash equilibrium if and only if the modified best-reply correspondences $\overline{\Phi}_I$ for $I\in\binom{[N]}{k}$ admit a simultaneous fixed point, that is $x\in E$ such that $x\in\overline{\Phi}_I(x)$ for all $I\in\binom{[N]}{k}$.
\end{theorem}
\begin{proof}
	Observe that the modified best-reply correspondences are fiberwise set-valued maps by construction.
	The claim evidently follows from Theorem~\ref{thm:SimultaneousFixedPointCriterion}.
\end{proof}

\medskip
The corollary of Theorem~\ref{thm:SimultaneousFixedPointCriterion} also carries over to families of games.

\begin{corollary}\label{cor:FamilyOfGamesSimultaneousCoincidenceWithDiagonal}
	Assume the notation introduced so far: A $k$-lateral Nash equilibrium of $\mathcal{E}/B$ exists if and only if
	\begin{displaymath}
		\Delta(E\times_B E)\cap \bigcap_{I\in \binom{[N]}{k}} \operatorname{Graph}(\overline{\Phi}_I) \neq \emptyset.
	\end{displaymath}
	Here $\Delta(E\times E):=\{(e,e)\colon e\in E\}$ denotes the diagonal in $E\times E$.
\end{corollary}
\begin{proof}
	We only argue the subtle difference to the proof of Corollary~\ref{cor:SimultaneousCoincidenceWithDiagonal}, since we are taking the fiberwise product of $E\times_B E$ in place of the product.
	The set-valued map $\overline{\Phi_I}\colon E\longrightarrow 2^E$ is fiberwise and so $\operatorname{Graph}(\overline{\Phi}_I)\subseteq E\times_B E$.
	The claim now follows from an observation that, by definition,
	\[
		\Delta(E\times_B E) = \bigsqcup_{b\in B}\Delta(E^{(b)}\times E^{(b)})
		\qquad\text{and}\qquad
		\operatorname{Graph}(\overline{\Phi}_I)  = \bigsqcup_{b\in B}\operatorname{Graph}(\overline{\Phi}_I^{(b)}).
	\]
	If $x=(e_1,e_2)\in \Delta(E\times_B E)\cap \bigcap_{I\in \binom{[N]}{k}} \operatorname{Graph}(\overline{\Phi}_I)$ then $e_1=e_2=:e$ for $e\in E^{(b)}$ where $b\in B$ is fixed.
	Hence, $e\in \Phi_I(e)$ is a fixed point of $\Phi_I$, by definition of the graph of a set-valued map, for every $I\in\binom{[N]}{k}$.
	This concludes the proof of the theorem.
\end{proof}

\medskip
Like in Section \ref{sec:familes_of_games} assume that $\Ecal/B:=(\Ecal^{(b)}:b\in B)$ is a family of games with the following properties
\begin{compactitem}[\quad --]
	\item $E_i$ is a fiber bundle over $B$, which is a subbundle of a finite dimensional (real) vector bundle $V_i$ over $B$ for all $1\leq i\leq N$,
	\item $E:=E_1\times_B\cdots \times_BE_N$ is the fiber product bundle over $B$ which is also a subbundle of the fiber product of vector bundle $V_1\times_B\cdots\times_B V_N\cong V_1\oplus\cdots\oplus V_N$,
	\item marginal functions $\phi_i \colon E \longrightarrow \mathbb{R}$ is continuous with respect to the inherited topology on $E$ for all $1\leq i\leq N$,
	\item best-reply correspondence $\Phi \colon E \longrightarrow 2^{E}$ is upper-semi continuous with respect to the topology on $E$.
\end{compactitem}

\medskip
In the footsteps of Section \ref{sec:familes_of_games} we derive further criterions for the existence of, now, multilateral Nash equilibria.
We begin by fixing a subset $I\in\binom{[N]}{k}$ and by considering again the pullback of the bundle $E\times_B E$ along the projection map $\pi\colon E\longrightarrow B$.
In other words there is a pullback square diagram
\begin{equation}\label{eq:pull-back-02}
	\xymatrix{
	\pi^*(E\times_B E)\ar[rr]^-{\Pi}\ar[d]^{\mu} & & E\times_B E\ar[d]^{\pi\times_B\pi}\\
	E\ar[rr]^-{\pi}& & B.
	}
\end{equation}
Since now the best-reply correspondence $\overline{\Phi}_I \colon E \longrightarrow 2^{E}$ is fiberwise it induces the set-valued section $\widetilde{\Phi}_I\colon E\longrightarrow 2^{\pi^*(E\times_B E)}$ of the bundle $\pi^*(E\times_B E)$ given by
\begin{displaymath}
	\widetilde{\Phi}_I (e):= \{e\}\times\{e\}\times  \overline{\Phi}_I(e) \ \subseteq \ \mu^{-1}(\{e\})  \ \subseteq \ \pi^*(E\times_B E).
\end{displaymath}
Like before, the fact that $\overline{\Phi}_I$ is a set-valued section means that $\mu\big(\overline{\Phi}_I (e)\big)=\{e\}$.

\medskip
Again the section $\delta\colon E\longrightarrow \pi^*(E\times_B E)$, $\delta(e)=(e,e,e)$, of the bundle $\pi^*(E\times_B E)$  induces the diagonal set-valued section $\Delta\colon E\longrightarrow 2^{\pi^*(E\times_B E)}$ given by $\Delta(e)=\{\delta(e)\}$.

\medskip
Now we get a criterion of the existence of multilateral Nash equilibria for families of games analogous to Theorem \ref{thm:criterion_01}.

\begin{theorem}\label{thm:criterion_011}
	Assume the notation introduced so far: There exists a parameter $b\in B$ such that the non-cooperative game $\Ecal^{(b)}$ from the family $\Ecal/B$ has a $k$-lateral Nash equilibrium if and only if there exists a point $e\in \pi^{-1}(\{b\})$ such that $\Delta(e)\cap \bigcap_{I\in\binom{[N]}{k}}\widetilde{\Phi}_I(e)\neq\emptyset$, or in other words
	\begin{displaymath}
		\bigcap_{I\in\binom{[N]}{k}}\big(\bigcup_{\epsilon\in E}\widetilde{\Phi}_I (\epsilon)\big) \ \cap \ \bigcup_{\epsilon\in E} \Delta(\epsilon)\neq \emptyset.
	\end{displaymath}
\end{theorem}
\begin{proof}
	According to Theorem \ref{thm:FamilyOfGamesSimultaneousFixedPointCriterion}: there exists a parameter $b\in B$ such that the non-cooperative game $\Ecal^{(b)}$ from the family $\Ecal/B$ has a $k$-lateral Nash equilibrium if and only if the modified best-reply correspondences $\overline{\Phi}_I$ for $I\in\binom{[N]}{k}$ admit a simultaneous fixed point, that is $e\in E$ such that $e\in\overline{\Phi}_I(e)$ for all $I\in\binom{[N]}{k}$.
	This means, there exists a point $e\in E$ such that $e\in \bigcap_{I\in\binom{[N]}{k}}\overline{\Phi}_I(e)$ and so $(e,e,e)\in \Delta(e)\cap \bigcap_{I\in\binom{[N]}{k}}\widetilde{\Phi}_I(e)$.

	\smallskip\noindent
	Conversely, if $(w,u,v)\in \Delta(e)\cap \bigcap_{I\in\binom{[N]}{k}}\widetilde{\Phi}_I(e)$ then $(w,u,v)=(e,e,e)$ for some $e\in E$ and additionally  $(w,u,v)\in   \{e'\}\times\{e'\}\times  \overline{\Phi}_I(e')$ for some $e'\in E$ and for all  $I\in\binom{[N]}{k}$.
	Consequently, $(e,e,e)\in \{e'\}\times\{e'\}\times   \overline{\Phi}_I(e')$ for all  $I\in\binom{[N]}{k}$, implying that $e=e'$ and $e\in  \overline{\Phi}_I(e)$ for all  $I\in\binom{[N]}{k}$.
	Thus, all modified best-reply correspondences $\widetilde{\Phi}_I$  have the common fixed point at $e\in E$.
	This completes the proof.
\end{proof}

\medskip
We  derive also an analogue of the criterion from Corollary \ref{cor:criterion_02} now in the context of multilateral Nash equilibrium.
Recall that we use the fiberwise map  $\Pi\colon \pi^*(E\times_B E)\longrightarrow E\times_B E$ from the pull-back diagram \eqref{eq:pull-back-02}.

\begin{corollary}\label{cor:criterion_021}
	Assume the notation introduced so far: There exists a parameter $b\in B$ such that the non-cooperative game $\Ecal^{(b)}$ from the family $\Ecal/B$ has a $k$-lateral Nash equilibrium if and only if
	\begin{displaymath}
		\bigcap_{I\in\binom{[N]}{k}}\Pi\big(\bigcup_{\epsilon\in E}\widetilde{\Phi}_I (\epsilon) \big) \ \cap \ \Pi\big(\bigcup_{\epsilon\in E} \Delta(\epsilon)\big)\neq \emptyset.
	\end{displaymath}
\end{corollary}
\begin{proof}
	The fact follows directly from the previous theorem and the following implications: If the point $(e',e'')\in \Pi\big(\bigcup_{\epsilon\in E} \Delta(\epsilon)\big)$ then $e'=e''=:e$, and if the point $(e,e)\in\{e'''\}\times \overline{\Phi}_I(e''')$ then $e=e'''$ and $e\in \overline{\Phi}_I(e)$.
\end{proof}

\medskip
Now, combining the criterion from Corollary  \ref{cor:criterion_021} and Lemma \ref{lem:intersection_lemma} we prove the first ever result on the existence of multilateral Nash equilibria for families of games.

\begin{theorem}\label{thm:Main_result_02}
	Let $p$ be a prime, $N\geq 1$ and $1\leq k\leq N$ integers, and let $\Ecal/B:=(\Ecal^{(b)}:b\in B)$ be a family of $N$-player non-cooperative games with the following properties
	\begin{compactitem}[\quad --]
		\item $B$ is a compact  $\mathbb{F}_p$-orientable manifold,
		\item $E_i$ is a fiber bundle over $B$ with compact total space $E_i$, which is a subbundle of a finite dimensional (real) vector bundle $V_i$ over $B$ for all $1\leq i\leq N$,
		\item $E:=E_1\times_B\cdots \times_BE_N$ is the fiber product bundle over $B$ which is also a subbundle of the fiber product of vector bundle $V:= V_1\times_B\cdots\times_B V_N\cong V_1\oplus\cdots\oplus V_N$,
		\item the best-reply correspondence $\Phi\colon E\longrightarrow 2^E$ of $\Ecal/B$ is compact valued,
		\item for every $I\in\binom{[N]}{k}$ the homomorphism in cohomology, induced by the projection map $\pi\times_B\pi\colon E\times_B E\longrightarrow B$,
		\[
			\xymatrix{
			H^{\dim(B)}(B;\mathbb{F}_p)\ar[rr]^-{(\pi\times_B\pi)^*} & & 	H^{\dim(B)}\big(\Pi\big(\bigcup_{e\in E}\widetilde{\Phi}_I (e) \big) ;\mathbb{F}_p\big)
			}
		\]
		is injective.
	\end{compactitem}

	\smallskip\noindent
	If the monomial $e(V)^{\binom{N}{k}}=e(V_1)^{\binom{N}{k}}\cdots e(V_N)^{\binom{N}{k}}$ in $\mathbb{F}_p$ Euler classes of vector bundles $V_1,\dots, V_N$ does not vanish in the cohomology $H^*(B;\mathbb{F}_p)$, then there exists a parameter $b\in B$ such that the non-cooperative game $\Ecal^{(b)}$ from the family $\Ecal/B$ has a  $k$-lateral Nash equilibrium.
\end{theorem}
\begin{proof}
	The result follows from Corollary \ref{cor:criterion_02} and Lemma \ref{lem:intersection_lemma} by taking
	\begin{compactitem}[\quad --]
		\item  $T_1,\dots,T_{\binom{N}{k}}$ to be the subspaces $\Pi\big(\bigcup_{e\in E}\widetilde{\Phi}_I (e) \big)$ of $V\times_N V$ for $I\in\binom{[N]}{k}$, and
		\item  $T_0$ to be the intersection of the disk bundle $D(V\times_N V)$ and the total bundle of the diagonal vector subbundle of $V\times_N V$.
	\end{compactitem}
	Here the disk bundle is chosen in such a way that in contains all the sets $\Pi\big(\bigcup_{e\in E}\widetilde{\Phi}_I (e) \big)$, for $I\in\binom{[N]}{k}$, in its interior.
	This possible due to the assumed compactness of $E_1,\dots,E_N$ and $B$.
	In this way we have achieved that $T_1,\dots,T_{\binom{N}{k}}$ are compact and that
	\[
		\xymatrix{
			H^{\dim(B)}(B;\mathbb{F}_p) \longrightarrow H^{\dim(B)}(T_i;\mathbb{F}_p)
		}
	\]
	is injective for all $1\leq i\leq \binom{N}{k}$.
\end{proof}

\medskip
A particular highlight of our theorem is that the assumptions on the individual games are rather permissive, especially compared to more classical assumptions in game theory. To illustrate that, we additionally give a corollary of the previous theorem which relies on assumptions on the profit functions of the family of games rather than on the cohomology of the graphs of its best reply-correspondences.

\begin{corollary}\label{cor:MainResultMultilateralCorollay}
	Let $p$ be a prime, $N\geq 1$ an integer, and let $\Ecal/B:=(\Ecal^{(b)}:b\in B)$ be a family of $N$-player non-cooperative games with the following properties
	\begin{compactitem}[\quad --]
		\item $B$ is a compact $\mathbb{F}_p$-orientable manifold,
		\item $E_i$ is a fiber bundle over $B$ with compact total space $E_i$, which is a subbundle of a finite dimensional (real) vector bundle $V_i$ over $B$ for all $1\leq i\leq N$ with each fiber being convex,
		\item $E:=E_1\times_B\cdots \times_BE_N$ is the fiber product bundle over $B$ and is also a subbundle of the fiber product of vector bundle $V:= V_1\times_B\cdots\times_B V_N\cong V_1\oplus\cdots\oplus V_N$,
		\item $\theta_i^{(b)}(\cdot, x_{-i})\colon E_i\longrightarrow\mathbb{R}$, is quasi-concave for every $b\in B$, every $i\in [N]$ and every $x_{-i}\in E_{-i}$.
	\end{compactitem}

	\smallskip\noindent
	If the monomial $e(V_1)^{\binom{N}{k}}\cdots e(V_N)^{\binom{N}{k}}$ in $\mathbb{F}_p$ Euler classes of vector bundles $V_1,\dots, V_N$ does not vanish in the cohomology $H^*(B;\mathbb{F}_p)$, then there exists a parameter $b\in B$ such that the non-cooperative game $\Ecal^{(b)}$ from the family $\Ecal/B$ has a Nash equilibrium.
\end{corollary}
\begin{proof}
	To simplify the notation we often denote restrictions of a map with the same symbol as the original map.
	In order to apply Theorem \ref{thm:Main_result_02} to our current situation we only need to prove that the map in cohomology
	\[
		\xymatrix{
		H^{\dim(B)}(B;\mathbb{F}_p)\ar[rr]^-{(\pi\times_B\pi)^*} & & 	H^{\dim(B)}\big(\Pi\big(\bigcup_{e\in E}\widetilde{\Phi}_I (e) \big) ;\mathbb{F}_p\big)
		}
	\]
	is injective for every $I\in\binom{[N]}{k}$.
	Consider the following derivative of the commutative diagram \eqref{eq:pull-back}:
	\begin{equation*}\label{eq:pull-back-03}
		\xymatrix{
		\bigcup_{e\in E}\widetilde{\Phi}_I (e)\ar[rr]^-{\Pi}\ar[d]^{\mu} & & \Pi\big(\bigcup_{e\in E}\widetilde{\Phi}_I (e) \big)\ar[d]^{\pi\times_B\pi}\\
		E\ar[rr]^-{\pi}& & B.
		}
	\end{equation*}
	Applying the cohomology functor $H^*(\cdot\ ;\mathbb{F}_p)$ we get the commutative diagram:
	\begin{equation*}\label{eq:pull-back-0e}
		\xymatrix{
		H^*(\bigcup_{e\in E}\widetilde{\Phi}_I (e);\mathbb{F}_p)  & &\ar[ll]_-{\Pi^*} H^{*}\big(\Pi\big(\bigcup_{e\in E}\widetilde{\Phi}_I (e) \big) ;\mathbb{F}_p\big)\\
		H^*(E;\mathbb{F}_p) \ar[u]_{\mu^*}& & \ar[ll]_{\pi^*} H^*(B;\mathbb{F}_p)\ar[u]_{(\pi\times_B\pi)^*}.
		}
	\end{equation*}
	Since $\Pi^*\circ ( \pi\times_B\pi)^*=\mu^*\circ\pi^*$, to prove that $( \pi\times_B\pi)^*$ is injective, it suffices to show that $\mu^*\circ\pi^*$ is injective.

	\smallskip
	The quasi-concavity of the functions $\theta_i^{(b)}(\cdot, x_{-i})\colon E_i\longrightarrow\mathbb{R}$, according to Lemma \ref{lem:MultilateralBestReplyConvex}, yields the convexity of the best-reply correspondence $\Phi_I(e)$ and, therefore, the set-valued section $\widetilde{\Phi}_I (e)$ for every $e\in E$.
	Consequently, the map $\mu$ is a Vietoris map which, by Vietoris' theorem \cite{Vietoris1927}, induces an isomorphism $\mu^*\colon H^*(E;\mathbb{F}_p)\longrightarrow H^*(\bigcup_{e\in E}\widetilde{\Phi}_I (e);\mathbb{F}_p)$ in all degrees.

	\smallskip
	On the other hand, since we assumed that the fibers in $E_1,\dots,E_N$ are convex, we have that the fibers of $E$, which is the fiber product of $E_1,\dots,E_N$ over $B$, are also convex, implying that $\pi$ is a Vietoris map as well.
	Thus, $\pi^*\colon H^*(B;\mathbb{F}_p)\longrightarrow H^*(E;\mathbb{F}_p)$ is an isomorphism.

	\smallskip
	In summary, we have proved that $( \pi\times_B\pi)^*$ is injective, and so we completed the proof.
\end{proof}

\medskip
The quasi-concavity assumption models the famous law of economics known as the law of diminishing marginal returns, which roughly states that for every additional unit of input, the increase in output must diminish.

\medskip
The previous result, Corollary \ref{cor:MainResultMultilateralCorollay}, can be improved further by relaxing the condition on the vector bundle $V\longrightarrow B$ as follows.
To accomplish that, we organize the $\binom{N}{k}$ set-valued sections $\widetilde{\Phi}_I\colon E\longrightarrow 2^{\pi^*(E\times_B E)}$ of the bundle $\pi^*(E\times_B E)$ in a particular way.

\medskip
First we observe the following  property of the modified best-reply correspondences $\overline{\Phi}_I$.
Let $I_1, I_2 \in\binom{[N]}{k}$ and assume that $I_1\cap I_2=\emptyset$.
Suppose $b\in B$ is arbitrary and additionally assume that $y\in \overline{\Phi}^{(b)}_{I_1}(x)$ and $y\in \overline{\Phi}^{(b)}_{I_2}(x)$ for some $x,y\in E^{(b)}$.
Then by definition of the modified best-reply correspondences we have that
\begin{align*}
	y\in \overline{\Phi}_{I_1}^{(b)}(x):= & \{(y_{I_1}, z_{-I_1})\in E^{(b)} : z_{-I_1}\in E^{(b)}_{-I_1}, \ \theta_{I_1}^{(b)}(y_{I_1}, x_{-I_1}) = \phi_{I_1}^{(b)}(x)\}, \\
	y\in \overline{\Phi}_{I_2}^{(b)}(x):= & \{(y_{I_2}, z_{-I_2})\in E^{(b)} : z_{-I_2}\in E^{(b)}_{-I_2}, \ \theta_{I_2}^{(b)}(y_{I_2}, x_{-I_2}) = \phi_{I_2}^{(b)}(x)\}.
\end{align*}
Since we have assumed that $I_1\cap I_2=\emptyset$ the following equivalence holds:
\[
	y\in \overline{\Phi}_{I_1}^{(b)}(x) \cap \overline{\Phi}_{I_2}^{(b)}(x)
\]
if and only if
\begin{displaymath}
	y\in \{(y_{I_1\cup I_2}, z_{-I_1\cup I_2}) : z_{-I_1\cup I_2} \in E^{(b)}_{-I_1\cup I_2}, \  \theta_{I_1}^{(b)}(y_{I_1}, x_{-I_1}) = \phi_{I_1}^{(b)}(x), \ \theta_{I_2}^{(b)}(y_{I_2}, x_{-I_2}) = \phi_{I_2}^{(b)}(x)\}.
\end{displaymath}

\medskip
This observation leads us to the following definition of the set-valued map
\[
	\overline{\Phi}_{\{I_1,I_2,\dots,I_r\}}^{(b)}(x) \colon E^{(b)}\longrightarrow 2^{E^{(b)}}
\]
associated to any collection $\{I_1, \dots, I_r\}\in\binom{[N]}{k}$ of pairwise disjoint $k$ element subsets of $[N]$:
\begin{multline*}
	\overline{\Phi}_{\{I_1,I_2,\dots,I_r\}}^{(b)}(x) := \bigcap_{m=1}^{r} \overline{\Phi}_{I_m}^{(b)}(x) = \\ \big\{(y_{\bigcup_{m=1}^r I_m}, z_{-\bigcup_{m=1}^r I_m}) : z_{-\bigcup_{m=1}^r I_m} \in E^{(b)}_{-\bigcup_{m=1}^r I_m}, \ (\forall 1 \leq m \leq r) \ \theta_{I_m}^{(b)}(y_{I_m}, x_{-I_m}) = \phi_{I_m}^{(b)}(x)\big\}.
\end{multline*}
This extends in the obvious way to the set-valued map $\overline{\Phi}_{I_1,I_2,\dots,I_r}\colon E \longrightarrow 2^{E}$ for the entire family by setting
\begin{displaymath}
	\overline{\Phi}_{\{I_1,I_2,\dots,I_r\}}(x) = \overline{\Phi}_{\{I_1,I_2,\dots,I_r\}}^{(b)}(x)
\end{displaymath}
for $x \in E^{(b)}\subseteq E$.
Note that by construction
\[
	\overline{\Phi}_{\{I_1,I_2,\dots,I_r\}}(x) = \bigcap_{m=1}^r \overline{\Phi}_{I_m}(x)
\]
for every $x\in E$.
Analogously to the previous constructions, the set-valued map we just introduced $\overline{\Phi}_{\{I_1,I_2,\dots,I_r\}}\colon E \longrightarrow 2^{E}$ induces the corresponding set-valued section $\widetilde{\Phi}_{\{I_1,I_2,\dots,I_r\}}\colon E\longrightarrow 2^{\pi^*(E\times_B E)}$ of the bundle $\pi^*(E\times_B E)$ by setting
\begin{displaymath}
	\widetilde{\Phi}_{\{I_1,I_2,\dots,I_r\}} (x):= \{x\}\times\{x\}\times  \overline{\Phi}_{\{I_1,I_2,\dots,I_r\}}(x) \ \subseteq \ \mu^{-1}(\{x\})  \ \subseteq \ \pi^*(E\times_B E).
\end{displaymath}

\medskip
Thus, the intersection from the simultaneous coincidence condition in Theorem~\ref{thm:criterion_011} and Corollary~\ref{cor:criterion_021} can be replaced with an intersection containing fewer sets, by grouping the modified best-reply correspondences as above.

\medskip
To get insight into how many sets will be left in the intersection after the grouping, we need to consider the following question:
What is the minimal cardinality $\xi(N,k)$ of a partition $\{\mathcal{P}_1,\dots,\mathcal{P}_{\ell}\} \subseteq 2^{\binom{[N]}{k}}$ of the family $\binom{[N]}{k}$ satisfying the property that any two $k$-element subsets of $[N]$ belonging to the same element of the partition, say $\mathcal{P}_i$, must be disjoint?
For such a partition $\{\mathcal{P}_1,\dots,\mathcal{P}_{\ell}\}$ of $\binom{[N]}{k}$ we have that
\begin{equation}\label{eq:BestReplyIntersection}
	\bigcap_{I\in \binom{[N]}{k}}\widetilde{\Phi}_I=\bigcap_{i=1}^{\ell}\bigcap_{I\in\mathcal{P}_i}\widetilde{\Phi}_{I}=\bigcap_{i=1}^{\ell} \widetilde{\Phi}_{\mathcal{P}_i},
\end{equation}
which is why we are interested in understanding what the minimal possible number $\xi(N,k)$ for the cardinality $\ell$ is.
It is evident that $\xi(N,1)=N$,  $\xi(N,k)\leq \binom{N}{k}$ and $\xi(N,k)= \binom{N}{k}$ for all $k\geq \lfloor\frac{N}{2}\rfloor+1$.
Furthermore, we have the following lemma.

\begin{lemma}\label{lem:value_of_xi}
	Let $N\geq 1$ and $1\leq k\leq N$ be integers.
	Then $\xi(N,k)$ coincides with the clique-covering number of the Kneser graph $K(N,k)$.
\end{lemma}

\medskip
Before arguing this lemma, we recall the relevant notions, for more insight, see for example the classical book \textit{Introduction to Graph Theory} by Douglas B. West from 2001~\cite[Def. 5.3.18]{DouglasWest2001}. We start with the Kneser Graph.

\medskip
Let $N\geq 1$ and $1\leq k\leq N$ be integers.
The Kneser Graph $K(N,k)$ is defined as follows:
\begin{compactitem}[\quad --]
	\item the set of vertices $V(K(N,k)) = \binom{[N]}{k}$ is the set of all $k$-element subsets of $[N]$, and
	\item for two vertices $I,J\in V(K(N,k))$ there is an edge between them, $\{I,J\}\in E(K(N,k))$, if and only if $I\cap J = \emptyset$.
\end{compactitem}
So in the Kneser Graph, exactly those subsets of $[N]$ are connected by an edge, which are disjoint.
For example, $K(N,1)$ is a complete graph on $N$ vertices and $K(N,N)$ is a graph with one vertex and no edges.

\medskip
Next, recall the definition of a clique in a graph.
A subgraph $C$ of the graph $G$ is called a clique of size $c\geq 1$, if it is (isomorphic to) a complete graph on $c$ vertices.
Thus, a clique in a Kneser graph is a collection of pairwise disjoint subsets in $\binom{[N]}{k}$.

\medskip
Finally, before we get back to Lemma~\ref{lem:value_of_xi}, recall what is the clique covering number.
The clique covering number of a graph $G$, denoted by $\theta(G)$, is the minimal number of cliques $C_1,\dots,C_{\theta(G)}$ of the graph $G$ needed to cover the vertices of $G$, i.e.,
\begin{displaymath}
	V(G) = V(C_1)\cup\cdots\cup V(C_{\theta(G)}).
\end{displaymath}
Now, a clique covering of $K(N,k)$ is exactly a partition of $\binom{[N]}{k}$ into collections of pairwise disjoint sets, as Lemma~\ref{lem:value_of_xi} claims.

\medskip
As noted already, $K(N,1)$ is the complete graph on $N$ vertices, and so its clique covering number is $1$.
This means that in the case $k=1$ it suffices to have just one member of the intersection~\eqref{eq:BestReplyIntersection}, which corresponds to the fact that in the classical case, the best-reply correspondences can be combined into just one ``global'' best-reply correspondence.

\medskip
Having in mind the just described reorganization of the $\binom{N}{k}$ set-valued sections $\widetilde{\Phi}_I$ of $\pi^*(E\times_B E)$ into $\xi(N,k)$ ``intersection'' set-valued sections of $\pi^*(E\times_B E)$, of the form $\widetilde{\Phi}_{\{I_1,I_2,\dots,I_r\}}$, we get a further improvement of Theorem \ref{thm:Main_result_02} and Corollary \ref{cor:MainResultMultilateralCorollay}.

\begin{theorem}\label{thm:Main_result_02_02}
	Let $p$ be a prime, $N\geq 1$ and $1\leq k\leq N$ integers, and let $\Ecal/B:=(\Ecal^{(b)}:b\in B)$ be a family of $N$-player non-cooperative games with the following properties
	\begin{compactitem}[\quad --]
		\item $B$ is a compact  $\mathbb{F}_p$-orientable manifold,
		\item $E_i$ is a fiber bundle over $B$ with compact total space $E_i$, which is a subbundle of a finite dimensional (real) vector bundle $V_i$ over $B$ for all $1\leq i\leq N$,
		\item $E:=E_1\times_B\cdots \times_BE_N$ is the fiber product bundle over $B$ which is also a subbundle of the fiber product of vector bundle $V:= V_1\times_B\cdots\times_B V_N\cong V_1\oplus\cdots\oplus V_N$,
		\item the best-reply correspondence $\Phi\colon E\longrightarrow 2^E$ of $\Ecal/B$ is compact valued,
		\item for every $I\in\binom{[N]}{k}$ the homomorphism in cohomology, induced by the projection map $\pi\times_B\pi\colon E\times_B E\longrightarrow B$,
		\[
			\xymatrix{
			H^{\dim(B)}(B;\mathbb{F}_p)\ar[rr]^-{(\pi\times_B\pi)^*} & & 	H^{\dim(B)}\big(\Pi\big(\bigcup_{e\in E}\widetilde{\Phi}_I (e) \big) ;\mathbb{F}_p\big)
			}
		\]
		is injective.
	\end{compactitem}

	\smallskip\noindent
	If the monomial $e(V)^{\xi(N,k)}=e(V_1)^{\xi(N,k)}\cdots e(V_N)^{\xi(N,k)}$ in $\mathbb{F}_p$ Euler classes of vector bundles $V_1,\dots, V_N$ does not vanish in the cohomology $H^*(B;\mathbb{F}_p)$, then there exists a parameter $b\in B$ such that the non-cooperative game $\Ecal^{(b)}$ from the family $\Ecal/B$ has a  $k$-lateral Nash equilibrium.
\end{theorem}

\medskip
The improvement of Corollary \ref{cor:MainResultMultilateralCorollay} is the following result.
The proof of the corollary we present below easily evolves into a proof of the previous theorem.

\begin{corollary}\label{cor:Main_result_01_02}
	Let $p$ be a prime, $N\geq 1$ an integer, and let $\Ecal/B:=(\Ecal^{(b)}:b\in B)$ be a family of $N$-player non-cooperative games with the following properties
	\begin{compactitem}[\quad --]
		\item $B$ is a compact $\mathbb{F}_p$-orientable manifold,
		\item $E_i$ is a fiber bundle over $B$ with compact total space $E_i$, which is a subbundle of a finite dimensional (real) vector bundle $V_i$ over $B$ for all $1\leq i\leq N$ with each fiber being convex,
		\item $E:=E_1\times_B\cdots \times_BE_N$ is the fiber product bundle over $B$ and is also a subbundle of the fiber product of vector bundle $V:= V_1\times_B\cdots\times_B V_N\cong V_1\oplus\cdots\oplus V_N$,
		\item $\theta_i^{(b)}(\cdot, x_{-i})\colon E_i\longrightarrow\mathbb{R}$, is quasi-concave for every $b\in B$, every $i\in [N]$ and every $x_{-i}\in E_{-i}$.
	\end{compactitem}

	\smallskip\noindent
	If the monomial $e(V_1)^{\xi(N,k)}\cdots e(V_N)^{\xi(N,k)}$ in $\mathbb{F}_p$ Euler classes of vector bundles $V_1,\dots, V_N$ does not vanish in the cohomology $H^*(B;\mathbb{F}_p)$, then there exists a parameter $b\in B$ such that the non-cooperative game $\Ecal^{(b)}$ from the family $\Ecal/B$ has a $k$-lateral Nash equilibrium.
\end{corollary}
\begin{proof}
	Let $\ell:=\xi(N,k)$.
	Choose an existing partition $\{\mathcal{P}_1,\dots,\mathcal{P}_{\ell}\}$ of the family $\binom{[N]}{k}$ which has the property that any two $k$-element subsets of $[N]$ belonging to the same element of partition, say $\mathcal{P}_i$, must be disjoint.

	\smallskip
	According to the criterion of Theorem \ref{thm:criterion_011}: There exists  $b\in B$ such that  $\Ecal^{(b)}$ from the family $\Ecal/B$ has a $k$-lateral Nash equilibrium if and only if there exists  $e\in \pi^{-1}(\{b\})$ such that
	\[
		\bigcap_{I\in\binom{[N]}{k}}\big(\bigcup_{\epsilon\in E}\widetilde{\Phi}_I (\epsilon)\big) \ \cap \ \bigcup_{\epsilon\in E} \Delta(\epsilon)\neq \emptyset.
	\]
	Since the sequence of equalities holds
	\[
		\bigcap_{I\in\binom{[N]}{k}}\big(\bigcup_{\epsilon\in E}\widetilde{\Phi}_I (\epsilon)\big)=
		\bigcup_{\epsilon\in E}\big( \bigcap_{I\in\binom{[N]}{k}}\widetilde{\Phi}_I (\epsilon)\big)=
		\bigcup_{\epsilon\in E}\big( \bigcap_{i=1}^{\ell}\bigcap_{I\in\mathcal{P}_i}\widetilde{\Phi}_{I} \big)=
		\bigcup_{\epsilon\in E}\big( \bigcap_{i=1}^{\ell} \widetilde{\Phi}_{\mathcal{P}_i}\big)=
		\bigcap_{i=1}^{\ell}\big(\bigcup_{\epsilon\in E} \widetilde{\Phi}_{\mathcal{P}_i}\big)
	\]
	the proof of Corollary \ref{cor:criterion_021} yields the following refined criterion: There exists a parameter $b\in B$ such that  $\Ecal^{(b)}$ from the family $\Ecal/B$ has a $k$-lateral Nash equilibrium if and only if
	\begin{displaymath}
		\bigcap_{i=1}^{\ell} \Pi\big(\bigcup_{\epsilon\in E} \widetilde{\Phi}_{\mathcal{P}_i}\big)\ \cap \ \Pi\big(\bigcup_{\epsilon\in E} \Delta(\epsilon)\big)\neq \emptyset.
	\end{displaymath}

	\smallskip
	Now, we mimic the proof of Corollary \ref{cor:MainResultMultilateralCorollay} using the criterion we just derived in place of Corollary \ref{cor:criterion_021}.
	In order to use Lemma \ref{lem:intersection_lemma}, like in the proof of Theorem \ref{thm:Main_result_02}, it suffices to prove that the map in cohomology
	\[
		\xymatrix{
		H^{\dim(B)}(B;\mathbb{F}_p)\ar[rr]^-{(\pi\times_B\pi)^*} & & 	H^{\dim(B)}\big(\Pi\big(\bigcup_{e\in E}\widetilde{\Phi}_{\mathcal{P}_i} (e) \big) ;\mathbb{F}_p\big)
		}
	\]
	is injective for every $1\leq i\leq \ell$.
	Again consider a derivative of the commutative diagram \eqref{eq:pull-back}:
	\begin{equation*}\label{eq:pull-back-04}
		\xymatrix{
		\bigcup_{e\in E}\widetilde{\Phi}_{\mathcal{P}_i} (e)\ar[rr]^-{\Pi}\ar[d]^{\mu} & & \Pi\big(\bigcup_{e\in E}\widetilde{\Phi}_{\mathcal{P}_i} (e) \big)\ar[d]^{\pi\times_B\pi}\\
		E\ar[rr]^-{\pi}& & B.
		}
	\end{equation*}
	The cohomology functor $H^*(\cdot\ ;\mathbb{F}_p)$ yields the commutative diagram:
	\begin{equation*}\label{eq:pull-back-05}
		\xymatrix{
		H^*(\bigcup_{e\in E}\widetilde{\Phi}_{\mathcal{P}_i} (e);\mathbb{F}_p)  & &\ar[ll]_-{\Pi^*} H^{*}\big(\Pi\big(\bigcup_{e\in E}\widetilde{\Phi}_{\mathcal{P}_i} (e) \big) ;\mathbb{F}_p\big)\\
		H^*(E;\mathbb{F}_p) \ar[u]_{\mu^*}& & \ar[ll]_{\pi^*} H^*(B;\mathbb{F}_p)\ar[u]_{(\pi\times_B\pi)^*}.
		}
	\end{equation*}
	Since $\Pi^*\circ ( \pi\times_B\pi)^*=\mu^*\circ\pi^*$, the injectivity of $( \pi\times_B\pi)^*$ follows from the injectivity of the composition $\mu^*\circ\pi^*$.
	Hence, we analyze the homomorphisms $\mu^*$ and $\pi^*$.

	\smallskip
	The quasi-concavity of the functions $\theta_i^{(b)}(\cdot, x_{-i})\colon E_i\longrightarrow\mathbb{R}$, according to Lemma \ref{lem:MultilateralBestReplyConvex}, yields the convexity of the best-reply correspondence $\Phi_I(e)$ and, therefore, the set-valued section $\widetilde{\Phi}_I (e)$ for every $e\in E$.
	Since $\widetilde{\Phi}_{\mathcal{P}_i} (e)=\bigcap_{I\in \mathcal{P}_i}\widetilde{\Phi}_I (e)$, and a non-empty intersection  f convex sets is always convex, the set-valued section $\widetilde{\Phi}_{\mathcal{P}_i}$ is also convex valued.
	Therefore, the map $\mu$ is a Vietoris map which, by Vietoris' theorem \cite{Vietoris1927}, induces an isomorphism $\mu^*\colon H^*(E;\mathbb{F}_p)\longrightarrow H^*(\bigcup_{e\in E}\widetilde{\Phi}_{\mathcal{P}_i} (e);\mathbb{F}_p)$ in all degrees.

	\smallskip
	On the other hand, as we have already seen, from convexity  of fibers in $E_1,\dots,E_N$ follows the convexity of fibers in $E$, which implies that $\pi$ is a Vietoris map as well.
	Hence, $\pi^*\colon H^*(B;\mathbb{F}_p)\longrightarrow H^*(E;\mathbb{F}_p)$ is an isomorphism.

	\smallskip
	In summary, $( \pi\times_B\pi)^*$ is injective, and so we completed the proof.
\end{proof}

\section{Examples}

In this section we present different examples of family of games which satisfy assumptions of Corollary \ref{cor:MainResultMultilateralCorollay} and/or  Corollary \ref{cor:Main_result_01_02}.

\subsection{The first simple example}
For an integer $d\geq 1$ let $\R P^d$ denote the real projective space of dimension $d$.
Set $V^d:=\{(L,v)\in \R P^d\times\R^{d+1} : v\in L\}$ to be the total space of the tautological vector bundle over $\R P^d$.

\medskip
Let $N\geq 1$ and $1\leq k\leq N$ be integers, and let the family of games $\Ecal/\R P^d$ be defined as follows.
Suppose each strategy space $E_i$ is a subbundle of the vector bundle $V^{d}$ with compact and convex fibers.
Assume that the profit functions $\theta_i^{(L)}(\cdot, x_{-i})\colon E_i\longrightarrow\mathbb{R}$ are quasi-concave for every $L\in \R P^d$, every $i\in [N]$ and every fixed $x_{-i}\in E_{-i}$.

\medskip
According to Corollary \ref{cor:Main_result_01_02} and the well known properties of the Stiefel--Whitney classes of tautological vector bundles over a real projective space \cite[Ch.\,4]{Milnor1974}, if $n\cdot \xi(n,k)\leq d$, then there exists a line  $L\in \R P^d$ such that the non-cooperative game $\Ecal^{(l)}$ from the family $\Ecal/\R P^d$ has a $k$-lateral Nash equilibrium.

\subsection{Finite Games}

Finite non-cooperative games have been considered ever since the original definition of a non-cooperative game by John Nash in his PhD thesis~\cite{NashPhD}. In this case, the word \textit{finite} refers to the initial assumption that the strategy spaces $E_1,\dots, E_N$ are finite sets.
Within the framework of finite games, these finitely many strategies, belonging to the global strategy space $E:=E_1\times\dots\times E_N$ are typically referred to as \textit{pure strategies}.

\medskip
However, not all finite games have equilibria, and the assumption of a finitely many strategies quickly turned out to be rather restrictive in terms of modelling possibilities.
The most well-known addition to finite games are the so-called mixed strategies.

\medskip
Before getting to common interpretations of mixed strategies, let us introduce them in a precise way.
Assume as usual that there are $N$ players and that player $i$ has $d_i$ pure strategies available to him.
Then his strategy space $E_i$ becomes the space of probability distributions $\pi^{(i)}$ on $d_i$ elements instead of taking for $E_i$ a finite set of cardinality $d_i$.
In other words, $E_i=\Delta_{d_i-1}:=\operatorname{conv}\{e_1,\dots,e_{d_i}\}$ is the standard simplex on $d_i$.
Here $e_1,\dots,e_{d_i}\in\mathbb{R}^{d_i}$ denote the canonical basis vectors in $\mathbb{R}^{d_i}$.
Now, the vertices of $\Delta_{d_i-1}$ represent pure strategies.
If a strategy $\pi^{(i)}\in\Delta_{d_i-1}$ also satisfies $\pi^{(i)}\in\operatorname{relint}\big(\operatorname{conv}\{e_{k_1},e_{k_2},\dots,e_{k_l}\}\big)$ for some subset $\{e_{k_1},e_{k_2},\dots,e_{k_l}\}$ of vertices, then we say that $\pi^{(i)}$ is a (genuine) mixture of pure strategies $\{e_{k_1},e_{k_2},\dots,e_{k_l}\}$.

\medskip
The profit functions, which were originally only defined for pure strategies, are extended to mixed strategies affinely.
More precisely, the payoffs of the pure strategies for player $i$ are recorded in $X_{i}\in\mathbb{R}^{d_1 + d_2 + \dots + d_N}$, the so-called pure strategy payoff (tensor) vector, which is given by
\begin{displaymath}
	X_{i}(j_1,\dots,j_N) = \theta_i(e_{j_1},\dots,e_{j_N}), \quad \text{for} \quad 1\leq j_k\leq d_k,\ 1\leq k\leq N.
\end{displaymath}
Then the affine extension of the pure-strategy profit functions is given as follows
\begin{equation}\label{eq:def:FiniteGameProfitFunction}
	\theta_i(\pi^{(1)}, \pi^{(2)}, \dots, \pi^{(N)}) := \sum_{j_1=1}^{d_1}\sum_{j_2=1}^{d_2}\cdots\sum_{j_N=1}^{d_N}\pi_{j_1}^{(1)}\cdot\pi_{j_2}^{(2)}\cdots\pi_{j_N}^{(N)}\cdot X_{i}(j_1,\dots,j_N),
\end{equation}
where $\pi^{(i)}\in\Delta_{d_i-1}\subseteq\mathbb{R}^{d_i}$. So
\begin{displaymath}
	\theta_i\colon \Delta_{d_1-1}\times \Delta_{d_2-1}\times\cdots\times\Delta_{d_N-1}\longrightarrow\mathbb{R}.
\end{displaymath}
Note that the profit functions are multilinear polynomials in $d_1 + d_2 + \cdots + d_N$ variables. If we treat $\theta_i$ as a function of $\pi^{(i)}$ and fix all other arguments $\pi^{(1)},\dots,\pi^{(i-1)},\pi^{(i+1)},\dots,\pi^{(N)}$, then $\theta_i$ is linear, since the coordinates of $\pi^{(i)}$ appear only once in every monomial.

\medskip
Mixed strategies originally appeared as an intermediate step in a proof that  finite games have pure-strategy Nash equilibrium, but on the other hand they also have realistic interpretations.
A good example to illustrate value of mixed strategies can be found in the well-known classical \textit{Inspection Game}.
Our presentation here is based on the classical book by Fudenberg and Tirole~\cite[Ex.\,1.7]{FudenbergTirole1991}.
The game models a common dilemma in real life which arises amongst boss $(B)$ and employee $(E)$: The boss hires an employee for a wage $w\in\mathbb{R}_{>0}$ in return for an amount of value $v\in\mathbb{R}$.
The employee has the option to work for a personal effort of $g>0$, where $g < w$, or not to work at no personal cost.
The boss has the option to inspect the employee at a cost $h>0$, where $h < w$ or not to inspect without cost.
If the employee is caught slacking, then the boss does not pay the employee the wage $w$.
The boss is not allowed to condition the payment of the wage on the output of the employee, and the employee and boss are assumed to make the choice simultaneously.
This situation is a $2$-player non-cooperative game with payoff (tensor) vector organised in the following matrices:
\begin{displaymath}
	X_{(E)} = \begin{pmatrix}
		0     & w     \\
		w - g & w - g
	\end{pmatrix}, \qquad X_{(B)} = \begin{pmatrix}
		-h    & -w  \\
		v-w-h & v-w
	\end{pmatrix}.
\end{displaymath}
Here the choice of the employee is not to work in the top row and to work in the bottom row, while the columns represent the choice of the boss, with the left column representing the boss opting to inspect the employee and the right column representing the boss opting not to inspect the employee.
It is quite easy to see that this game does not have a pure-strategy Nash equilibrium, simply by observing that in every case one player would prefer to change his strategy.

\medskip
Nonetheless, there does exist a mixed-strategy Nash equilibrium.
In this model, the mixed strategy represents the assumption that the same game is played multiple times without memory and that both players randomize their strategy.
When allowing the players to randomize their strategies, Fudenberg and Tirole show that the (totally mixed) Nash equilibrium of the Inspection game is given by the employee not working with probability $g/w$ and the boss inspecting with probability $h/w$.
Note that curiously, the equilibrium point does not depend on the value $v$ of the work.

\medskip
In general, a finite game is completely defined by the integers $d_1,\dots,d_N\in\mathbb{N}$ and the payoff vectors $X_1,\dots,X_N\in\mathbb{R}^{d_1+\dots+ d_N}$. So set $B_N:= \{(\mathbf{d}, \mathbf{X})\colon \mathbf{d}\in\mathbb{N}^N, \mathbf{X}\in\mathbb{R}^{d_1+\dots+d_N}\}$.
Then all finite $N$-player games constitute a family of games over the parameter space $B_N$, which we denote by $\mathcal{F}/B_N$. We have for $b=(\mathbf{d},\mathbf{X})\in B_N$ and $1\leq i\leq N$ that $\mathcal{F}^{(b)}_i = \Delta_{d_i}$ and the profit function of player $i$ is defined in~\eqref{eq:def:FiniteGameProfitFunction}.

\medskip
Since the breakthrough of Nash, the study of finite (non-cooperative) games expanded in many directions, which inspired further extensions, generalizations, applications and ultimately shaped the entirety of game theory in a profound way.
For that reason it is not possible to give a concise, complete and comprehensive history of the research related to finite games.
Instead, we recall results relevant to the study of existence of multilateral Nash equilibria.

\medskip
The one line of research was motivated by the question: How many Nash equilibria can a finite game have? Or even better, what is the set of all equilibria?
The famous result of Nash \cite{Nash1951} says that a finite game has at least one equilibrium.
Wilson \cite[Thm.\,1]{Wilson1971} in 1971 showed that a ``generic'' finite game always has an odd number of Nash equilibria.
In parallel, Harsanyi in his work \cite{Harsanyi1973} from 1973 showed that every finite game with a ``generic'' payoff vector has finitely many equilibria.
Here generic means all payoff vector except a set whose closure has Lebesgue measure zero.
The tight upper bound for the number of equilibria for the games with the generic payoff was given by McKelvey and McLennan in \cite{McKelveyMcLennan1997}.
The study of ``non-generic'' finite games was featured, for example, in works of Sturmfels \cite{Sturmfels2002}, Datta \cite{Datta2003a,Datta2003b,Datta2010}, and recently in work of Portakal \& Sturmfels \cite{PortakalSturmfels2022}, Abo, Portakal \& Sodomaco \cite{AboPortakalSodomaco2025}.
Amongst many results, we recall only the result \cite[Thm.\,1]{Datta2003b} which says:
\begin{theorem*}
	Every real algebraic variety is isomorphic to the set of (totally mixed) Nash equilibria of some $3$-person game, and also of an $N$-person game in which each player has two pure strategies.
\end{theorem*}

\medskip
Since a ``generic'' finite game has finitely many Nash equilibria it is highly unlikely that any of them is a multilateral one.
On the other hand, by the quoted result of Datta \cite[Thm.\,1]{Datta2003b}, any real algebraic variety is a set of Nash equilibria of some finite game.
Thus, it is to be expected that the filtration of the variety by multilateral equilibria \eqref{eq:filtration} is non-trivial.
Hence, it is natural to ask: Does this family $\mathcal{F}/B_N$ contain games with $k$-lateral equilibria?

\medskip
Having in mind the result of Datta on the so-called universality of $3$-person games, in the following we construct a $3$-player finite game which has $2$-lateral Nash equilibrium.

\medskip
To simplify our presentation, we borrow the notation of Sturmfels \cite{Sturmfels2002} and denote the $3$ players by $A$, $B$, and $C$.
All strategy spaces coincide $E_A=E_B=E_C=\Delta_1$ with the $1$-dimensional standard simplex $\Delta_1=\{(y_1,y_2)\in\mathbb{R}^2 : y_1\geq 0, y_2\geq 0, y_1+y_2=1\}$.
Denote the choices of the players $A$, $B$, $C$ with $a$, $b$, $c$ respectively.
Now, we construct the profit vector
\begin{equation}\label{eq:profit_vector}
	\mathcal{X}:=(X_A(i,j,k), X_B(i,j,k), X_B(i,j,k) : i,j,k\in \{0,1\})\ \in \mathbb{R}^{3\cdot 8}
\end{equation}
such that the game has a $2$-lateral equilibrium.
The main difficulty, compared to the classical Nash equilibrium, arises from the fact that for a $2$-lateral Nash equilibrium we are no longer optimizing linear functions but instead quadratic functions.

\medskip
Indeed, the strategy $(a,b,c)\in \Delta_1\times\Delta_1\times\Delta_1$ is a $2$-lateral Nash equilibrium if for all $u=(u_1,u_2)\in\Delta_1$ and all $v=(v_1,v_2)\in \Delta_1$ the following inequalities are satisfied

\begin{align*}
	\sum_{i,j,k\in \{0,1\}} X_A(i,j,k) a_ib_jc_k & \geq \sum_{i,j,k\in \{0,1\}} X_A(i,j,k) u_iv_jc_k, \\
	\sum_{i,j,k\in \{0,1\}} X_B(i,j,k) a_ib_jc_k & \geq \sum_{i,j,k\in \{0,1\}} X_B(i,j,k) u_iv_jc_k, \\
	\sum_{i,j,k\in \{0,1\}} X_A(i,j,k) a_ib_jc_k & \geq \sum_{i,j,k\in \{0,1\}} X_A(i,j,k) u_ib_jv_k, \\
	\sum_{i,j,k\in \{0,1\}} X_C(i,j,k) a_ib_jc_k & \geq \sum_{i,j,k\in \{0,1\}} X_C(i,j,k) u_ib_jv_k, \\
	\sum_{i,j,k\in \{0,1\}} X_B(i,j,k) a_ib_jc_k & \geq \sum_{i,j,k\in \{0,1\}} X_B(i,j,k) a_iu_jv_k, \\
	\sum_{i,j,k\in \{0,1\}} X_C(i,j,k) a_ib_jc_k & \geq \sum_{i,j,k\in \{0,1\}} X_C(i,j,k) a_iu_jv_k.
\end{align*}
Equivalently, for all $(u_1,u_2),(v_1,v_2)\in \Delta_1$ it should hold that
\begin{align}\label{ineq:2lateralCondition}
	\sum_{i,j,k\in \{0,1\}} X_A(i,j,k) c_k(a_ib_j - u_iv_j)  & \geq 0, \nonumber \\
	\sum_{i,j,k\in \{0,1\}} X_B(i,j,k) c_k(a_ib_j - u_iv_j)  & \geq 0, \nonumber \\
	\sum_{i,j,k\in \{0,1\}} X_A(i,j,k) b_j(a_ic_k - u_iv_k)  & \geq 0, \nonumber \\
	\sum_{i,j,k\in \{0,1\}} X_C(i,j,k) b_j(a_ic_k - u_iv_k) & \geq 0, \nonumber \\
	\sum_{i,j,k\in \{0,1\}} X_B(i,j,k) a_i(b_jc_k - u_jv_k) & \geq 0, \nonumber \\
	\sum_{i,j,k\in \{0,1\}} X_C(i,j,k) a_i(b_jc_k - u_jv_k) & \geq 0.
\end{align}

\medbreak
Now, let us inspect whether there exists a profit vector \eqref{eq:profit_vector}, or in other words $3$-player finite game, such that $a=b=c=(0,1)\in\Delta_1$ is a $2$-lateral Nash equilibrium of the finite game associated to this profit vector.
Hence, choosing $a$, $b$ and $c$ to be constant, the system of inequalities \eqref{ineq:2lateralCondition} becomes
\begin{align}\label{ineq:02}
	-X_A(1,1,2)u_1v_1  -X_A(1,2,2)u_1v_2 -X_A(2,1,2)u_2v_1+X_A(2,2,2)(1-u_2v_2) & \geq 0,\nonumber \\
	-X_B(1,1,2)u_1v_1  -X_B(1,2,2)u_1v_2 -X_B(2,1,2)u_2v_1+X_B(2,2,2)(1-u_2v_2) & \geq 0,\nonumber \\
	-X_A(1,2,1)u_1v_1  -X_A(1,2,2)u_1v_2 -X_A(2,2,1)u_2v_1+X_A(2,2,2)(1-u_2v_2) & \geq 0,\nonumber \\
	-X_C(1,2,1)u_1v_1  -X_C(1,2,2)u_1v_2 -X_C(2,2,1)u_2v_1+X_C(2,2,2)(1-u_2v_2) & \geq 0,\nonumber \\
	-X_B(2,1,1)u_1v_1  -X_B(2,1,2)u_1v_2 -X_B(2,2,1)u_2v_1+X_B(2,2,2)(1-u_2v_2) & \geq 0,\nonumber \\
	-X_C(2,1,1)u_1v_1  -X_C(2,1,2)u_1v_2 -X_C(2,2,1)u_2v_1+X_C(2,2,2)(1-u_2v_2) & \geq 0.
\end{align}
It is evident that the set of profit vectors \eqref{eq:profit_vector} in $\mathbb{R}^{3\cdot 8}$ which satisfy inequalities \eqref{ineq:02} is non-empty and it is also clearly infinite.

\medskip
In summary, in the case of finite games, the phenomena of multilateral equilibria are not rare and should be studied systematically.
Consequently, all questions posed and considered for Nash equilibria of finite games now get a new life in the context of multilateral Nash equilibria.

\subsection{Microeconomy}

Inspired by the classical Nobel-prize winning work of Kenneth Arrow and G\'erard Debreu \cite{ArrowDebreu1954}, we build a model of the state economy consisting of $N$ producers and $M$ consumers.
For a similar approach to microeconomy, also compare the classical books by Henderson and Quandt~\cite{HendersonQuandt1958} and Baumol~\cite{Baumol1965}.
However, the models presented in these books previous economic models mostly dealt with producers, consumers and the state separately. Instead, Arrow and Debreu unified the three economic categories into a single model, which they called \textit{the} Abstract Economy. Our players also consist of the producers and consumers, but the state is given a different task.
We develop a family of $N+M$-player games parameterized by the Cartesian product of two infinite Grassmanians $G_n(\mathbb{R}^{\infty})\times G_m(\mathbb{R}^{\infty})$ as our model. For more background information on the topological properties of the Grassmanian, see~\cite[Sec.\,5]{Milnor1974}. 

\medskip
The role of the producers in an economy is to create the commodities sold on the market. Any company or individual selling any commodity (we treat services as commodities) is a producer in the eyes of economy. The commodities are created from certain inputs, collectively referred to as \textit{production factors}. Examples of production factors include raw materials, infrastructure, labor, entrepreneurship, etc. We do not concern ourselves with the specific economic theory of production factors, rather we just assume that all possible production factors $(\xi_i : i\in\mathbb{N})$ are enumerated. We choose the symbol $\xi_i$ because the classical notation of $x_i$ for production factors is already reserved for the strategy of player $i$.

The consumers within the economy purchase the commodities and consume them with the goal of satisfying their personal needs and gaining satisfaction. For example, consumers may purchase food, housing, and energy, but also luxury or convenience items and services. All the possible benefits which the consumer gains from the commodities he purchases are collectively referred to as \textit{utility} and the goal of every consumer is to maximize his utility. 
Let the possible commodities $(q_i :i\in\mathbb{N})$ also be enumerated. Here we use the classical notation of $q_i$ for commodity $i$.
In particular, for simplicity reasons, we are making an assumption that there is a countably infinite number of possible products and a countably infinite number of possible commodities.

\medskip
Each producer $i\in[n]$ has a finite budget, so he can only afford a bounded subset of $\mathbb{R}^{\infty}$ of production factors $\mathcal{B}_i\subseteq\mathbb{R}^{\infty}$.
We make the assumption that each $\mathcal{B}_i$ is convex, since if the producer can afford two combinations of production factors, then he can also afford any convex combination.
Additionally, without loss of generality, we assume that $0\in \operatorname{int}(\mathcal{B}_i)$ for every $i\in[n]$.

\medskip
In our model the state is responsible for ensuring the availability of production factors and for planning the production cycle, so it is somewhat akin to socialism. The state can do this by building infrastructure, urbanistic design, political and economic planning, legislature, etc.
The task of the state is reflected in the model by supposing that the state selects an $n$-dimensional subspace $U\subseteq \mathbb{R}^{\infty}$ of production factors.
We assume that the producer is allowed to buy any linear combination of available production factors within his budget, so the choice of the producer $i$ is an element of the set $E_i^U:=U\cap\mathcal{B}$.

\medskip
Similarly, we assume that the state can plan production for an $m$-dimensional subspace $V\subseteq\mathbb{R}^{\infty}$ of commodities.
The consumers have finite budgets, so say that consumer $j$ for $n+1\leq j\leq n+m$ has budget space $\mathcal{C}_j\subseteq\mathbb{R}^{\infty}$.
By the same logic as before, we may assume that $\mathcal{C}_j$ is convex and  $0\in \operatorname{int}(\mathcal{C}_j)$.
Each consumer is allowed to buy any linear combination of commodities he can afford, so consumer $j$ strategy space is $E_j^V:=\mathcal{C}_j\cap V$.

Note that in our model the state does not have the power to set market prices. This is arguably more realistic compared to the original Abstract Economy, where it was assumed that the state set market prices. Instead, we will use a well-known, classical economic model of market prices.

\medskip
First, let us deal with the producers.
The production function of producer $i$ for commodity $j$ is given as $f_{i,j}\colon \mathbb{R}^{\infty}\longrightarrow[0,\infty)$.
It depends on the production factors $\xi^{(i)}:=(\xi_{j}^{(i)} : j\in\mathbb{N})$ taken as input by producer $i$.
We assume that all production factors are completely consumed during production and that the producer never incurs losses due to wasted production factors.
The production function $f_{i,j}$ is assumed to be continuous and twice differentiable, and it must satisfy the \emph{Law of diminishing marginal productivity} of production factors~\cite[Sec.\,3.1]{HendersonQuandt1958}.
This means that the production function must be concave.
Furthermore, with no production factors, the producer cannot produce, which means that necessarily $f_{i,j}(0)=0$. And finally, $f_{i,j}$ is non-decreasing in every coordinate, meaning if the producer obtains more of a particular resource, all other resources being fixed, then he can also produce more.
As a consequence, all production functions $f_{i,j}$ are concave and nondecreasing in every coordinate, assuming the remaining coordinates are fixed.

\medskip
The producer obtains his revenue by selling the commodities he produces at the market prices.
Denote the market price of commodity $j$ by $p_j$. We assume that $p_j$ obeys the \emph{Cournot competition model} assumptions \cite[Sec.\,6.2]{HendersonQuandt1958}.
Originally, the Cournot competition assumed that the price function $p_j\colon[0,\infty)\longrightarrow[0,\infty)$ of commodity $j$ is a strictly monotonously decreasing, concave function of the state's total production of commodity $j$, which was intended to represent falling prices as demand for a particular commodity is satisfied.
Since our model contains the consumers, unlike the original Cournot competition, we can accurately model that the market price is a strictly monotonously decreasing, concave function of the total demand of a particular commodity as given by consumers' consumption behavior. Since the total demand for commodity $j$ is simply the sum of all consumer's consumption of that commodity $\sum_{k=n+1}^{n+m} q^{(k)}_j$, we conclude that, consequently, producer $i$'s revenue $R_i$ is given by
\begin{displaymath}
	R_i = \sum_{j\in\mathbb{N}} p_j\cdot f_{i,j}
\end{displaymath}
Recall that since the producer only produces commodities from the finite-dimensional subspace $U$, all but finitely many production functions must be zero. This means that the sum in the expression for the revenue $R_i$ is actually finite. We will make it even more precise in a moment after explaining the costs incurred by the producer.

\medskip
Every producer $i$ has a fixed cost independent of his production level, denoted $\operatorname{TFC}_i$ for \textit{total fixed cost}, which accounts for building maintenance, accounting costs, and similar expenses.
Furthermore, each producer $i$ must pay for the production factors he is taking as inputs.
These costs are accounted for by the function $\operatorname{TVC}_i\colon \mathbb{R}^{\infty}\longrightarrow\mathbb{R}$, which stands for \textit{total variable cost}.
Since the prices of the production factors are assumed constant, the functions $\operatorname{TVC}_i$ are linear, given as linear combinations of the quantities of production factors taken as inputs.
The profit of producer $i\in[n]$ is the difference between his revenue and his costs, which is given by
\begin{align*}
	\theta_i^{(U,V)} \colon E_1^U\times\dots\times E_n^U\times E_{n+1}^V\times   E_{n+m}^V\longrightarrow \  & \mathbb{R},                                                                             \\
	(\xi^{(1)},\dots,\xi^{(n)},q^{(n+1)},\dots,q^{(n+m)}) \longmapsto \                                      & \sum_{j\in\mathbb{N}} f_{i,j}(\xi^{(i)})\cdot p_j\Big(\sum_{k=n+1}^{n+m} q^{(k)}_j\Big) \\ & \qquad\qquad\qquad\qquad - \operatorname{TVC}_i(\xi^{(i)}) - \operatorname{TFC}_i.
\end{align*}
According to the product rule, $\theta_i$ is concave as the sum of concave functions.

\medskip
Next, we consider the consumers. We model the total satisfaction of a consumer $n+1\leq j\leq n+m$ using a utility function $u_j$. Since the consumer derives the utility from the commodities they purchased, $u_j\colon \mathbb{R}^{\infty}\longrightarrow \mathbb{R}$ depends on the commodity combination $q^{(j)}$ he consumes.
Strictly speaking, it is not necessary for the consumer to be aware of the specific values of his exact utility function.
It is enough to require that the consumer is able to consistently and continuously decide which of two possible commodity combinations he prefers, or whether he is indifferent.
These decisions induce a preorder on the space of all possible commodity combinations $\mathbb{R}^{\infty}$.
Under these assumptions, Debreu's representation theorem~\cite{Debreu1954} guarantees the existence of a continuous utility function $u_j$ which exactly reflects the consumer's preferences.

\medskip
According to the \emph{Law of diminishing marginal utility}~\cite[Sec.\,2.2]{HendersonQuandt1958} we have that the utility functions are concave
for every consumer $n+1\leq j\leq n+m$.
Thus, the utility functions are also concave in every coordinate.
Finally, like Arrow and Debreu we also allow consumer $j$ to have shares in producer $i$, in which case he gains a percentage $\alpha_{j,i}\in[0,1]$ of his profit. Similarly we allow consumers to share their utility, meaning a purchase by one consumer may to some extent increase the utility of another. 
If consumer $j$ is benefited by consumer $k$'s utility, we use the same notation $\alpha_{j,k}\in[0,1]$ for the percentage increase in $j$'s utility.
In the classical literature these effects are called \textit{external economies}.

Finally, we can express the consumer $j$ profit function as follows:
\begin{align*}
	\theta_j^{(U,V)} \colon E_1^U\times\dots\times E_n^U\times E_{n+1}^V\times   E_{n+m}^V\longrightarrow \  & \mathbb{R},                                                                                     \\
	(\xi^{(1)},\dots,\xi^{(n)},q^{(n+1)},\dots,q^{(n+m)}) \longmapsto \                                      &
	u_j(q^{(j)}) + \sum_{k=n+1}^{n+m} \alpha_{j,k}u_k(q^{(n+1)},\dots,q^{(n+m)})                                                                                                                               \\
	                                                                                                         & \qquad\qquad\qquad\qquad\qquad +\sum_{k=1}^{n} \alpha_{j,k}\theta_k(\xi^{(1)},\dots,\xi^{(n)}).
\end{align*}
Once again, since each of the summands in the expression for $\theta_j$, where $n+1\leq j\leq n+m$, are concave so is $\theta_j$.

\medskip
In summary, we have defined a family of non-cooperative games over the base space $B= G_n(\mathbb{R}^{\infty})\times G_m(\mathbb{R}^{\infty})$ with $N:=n+m$ players,
\begin{compactitem}[\quad --]
	\item the strategy spaces $E_1,\dots,E_n$ being subsets of the total space of the vector bundle $V_1=\cdots=V_n$, which is the pull-back of tautological vector bundle $E(\gamma^n)\longrightarrow G_n(\mathbb{R}^{\infty})$ with respect to the projection map $G_n(\mathbb{R}^{\infty})\times G_m(\mathbb{R}^{\infty})\longrightarrow G_n(\mathbb{R}^{\infty})$,
	\item the strategy spaces $E_{n+1},\dots,E_{n+m}$ being subsets of the total space of the vector bundle $V_{n+1}=\cdots=V_{n+m}$, which is the pull-back of tautological vector bundle $E(\gamma^m)\longrightarrow G_m(\mathbb{R}^{\infty})$ with respect to the projection map $G_n(\mathbb{R}^{\infty})\times G_m(\mathbb{R}^{\infty})\longrightarrow G_m(\mathbb{R}^{\infty})$, and
	\item fibers of strategy spaces being non-empty convex and compact.
\end{compactitem}
With the profit functions being concave all the assumptions of Corollary~\ref{cor:Main_result_01_02} are satisfied except for $B$ being a finite dimensional orientable manifold.
To resolve this issue we proceed as follows.

\medskip
Let $k$ be an integer with $1\leq k\leq N$ and let $d\gg 1$ be a very big integer which will depend on integer parameters $n$, $m$, and $k$.
Consider an ``approximation'' of $B$ by its subspace $B_d:=G_n(\mathbb{R}^{d})\times G_m(\mathbb{R}^{d})$, and let $V_1^d=\cdots=V_n^d$ and $V_{n+1}^d=\cdots=V_{n+m}^d$ be the restriction vector bundles of $V_1=\cdots=V_n$ and $V_{n+1}=\cdots=V_{n+m}$ to $B_d$, respectively.
Now, $B_d$ is a finite dimensional $\mathbb{F}_2$-orientable manifold.
Then, according to Corollary~\ref{cor:Main_result_01_02}, if the monomial
\begin{multline*}
	e(V_1^d)^{\xi(N,k)}\cdots e(V_N^d)^{\xi(N,k)}= \\ w_n(V_1^d)^{\xi(N,k)}\cdots w_n(V_n^d)^{\xi(N,k)}w_m(V_{n+1}^d)^{\xi(N,k)}\cdots w_n(V_{n+m}^d)^{\xi(N,k)}\neq 0
\end{multline*}
does not vanish in $H^*(B_d;\mathbb{F}_2)$, then there exists $(U,V)\in B_d$ such that the associated non-cooperative game has a $k$-lateral Nash equilibrium.
Here $w_i(V)$ denotes the $i$th Stiefel--Whitney class of the vector bundle $V$, and it is a well known fact that $w_{\dim(V)}(V)$ coincides with the $\mathbb{F}_2$ Euler class of $V$.

\medskip
Recall that the cohomology of $B$ can be described as a polynomial algebra without any truncations, that is
\[
	H^*(B;\mathbb{F}_2)\cong \mathbb{F}_2[w_1(\gamma^n),\dots, w_n(\gamma^n), w_1(\gamma^m),\dots, w_m(\gamma^m)],
\]
and that  $H^*(B_d;\mathbb{F}_2)$ ``approximates'' $H^*(B;\mathbb{F}_2)$.
Consequently, for every $n$, $m$, $1\leq k\leq n+m$, with $N=m+n$, there exists an integer $d\gg 1$ such that
\[
	w_n(V_n^d)^{\xi(N,k)}w_m(V_{n+1}^d)^{\xi(N,k)}\cdots w_n(V_{n+m}^d)^{\xi(N,k)}\neq 0.
\]
Now, according to Corollary~\ref{cor:Main_result_01_02}, with these parameters, there exists $(U,V)\in B_d\subseteq B$ such that the associated non-cooperative game has a $k$-lateral Nash equilibrium. Note that $k$ was arbitrary.

\medskip
Economically speaking, this means that there exists a subspace of available production factors $V^*$ and a subspace of commodities with planned production $U^*$ which guarantee that the state's economy admits an arbitrarily high-lateral Nash equilibrium.
Thus, if the state is able to ensure the availability of $V^*$ production factors and plan for the production of $U^*$ commodities, it will ensure the existence of an arbitrarily high-lateral equilibrium in its economy.
A higher-lateral equilibrium, by definition, disincentivizes even larger groups of players from deviating from the equilibrium strategy.
For the state, this would mean that no rational group of players would deviate from the equilibrium strategy, ensuring stability and predictability in the economy.

\end{document}